\documentclass[reqno]{amsart}

\usepackage{amssymb}
\usepackage{multirow}
\usepackage[all,cmtip]{xy}
\usepackage{tikz}

\newcommand{\ZZ}{\mathbb{Z}}
\newcommand{\NN}{\mathbb{N}}

\newcommand{\FF}{\mathbb{F}}

\theoremstyle{plain}
\newtheorem{thm}{Theorem}[section]
\newtheorem{lemm}[thm]{Lemma}
\newtheorem{prop}[thm]{Proposition}

\newtheorem{dfn}[thm]{Definition}

\makeatletter
\def\@seccntformat#1{\protect\textup{\protect\@secnumfont\expandafter\protect\csname format#1\endcsname\csname the#1\endcsname\protect\@secnumpunct}}

\makeatother
\usepackage[square,numbers]{natbib}

\usepackage{hyperref}
\usepackage{breakurl}
\usepackage{url}
\hypersetup{pdfpagemode=none,colorlinks,citecolor=blue,filecolor=black,linkcolor=black,urlcolor=blue}
\newcommand{\hide}[1]{\iffalse #1\fi}
\newcommand{\vct}[1]{\mathbf{#1}}
\usepackage{amsmath}
\newcommand{\zint}{\,..\,}
\newcommand{\nth}{^{\text{th}}}
\DeclareMathOperator{\moddec}{mod}
\renewcommand{\mod}[1]{\,(\moddec #1)}
\usepackage{enumitem}
\usepackage{nameref}
\makeatletter
\newcommand{\clabel}[2]{\protected@write \@auxout {}{\string \newlabel {#1}{{#2}{\thepage}{#2}{#1}{}} }\hypertarget{#1}{}}
\makeatother
\theoremstyle{remark}
\newtheorem{remark}[thm]{Remark}
\newcommand{\lb}{\allowbreak}
\usepackage{etoolbox}
\newcommand{\breaklist}[2][,\lb]{\def\nextitem{\def\nextitem{#1}}\renewcommand*{\do}[1]{\nextitem{##1}}\docsvlist{#2}}
\DeclareMathOperator{\cirdec}{circ}	
	
\renewcommand{\cir}[1]{\cirdec(\breaklist{#1})}

\newcommand{\pcir}[2]{\cirdec_{#1}(\breaklist{#2})}

\newcommand{\floor}[1]{\lfloor#1\rfloor}
\usepackage{empheq}
\newcommand{\BM}[1]{\boldmath {$#1$}\unboldmath}
\DeclareMathOperator{\autdec}{Aut}
\newcommand{\aut}[1]{\autdec(#1)}
\usepackage{booktabs}
\usepackage{upgreek}
\newcommand{\vctg}[1]{\boldsymbol{#1}}
\makeatletter
\renewcommand*\env@matrix[1][*\c@MaxMatrixCols c]{\hskip -\arraycolsep\let\@ifnextchar\new@ifnextchar\array{#1}}
\makeatother
\usepackage[font=small,labelfont=bf,textfont=normalfont,margin=0cm]{caption}
\newcommand{\hip}{\overline}
\numberwithin{equation}{section}
\newcommand{\geteq}{\tag{\theequation}\stepcounter{equation}}
\begin{document}

\title[]{Quaternary Hermitian self-dual codes of lengths 26, 32, 36, 38 and 40 from modifications of well-known circulant constructions}

\author{A. M. Roberts}

\address{Department of Mathematical and Physical Sciences\\
Thornton Science Park\\
University of Chester\\
England}
\email{adammichaelroberts@outlook.com}


\begin{abstract}
In this work, we give three new techniques for constructing Hermitian self-dual codes over commutative Frobenius rings with a non-trivial involutory automorphism using $\lambda$-circulant matrices. The new constructions are derived as modifications of various well-known circulant constructions of self-dual codes. Applying these constructions together with the building-up construction, we construct many new best known quaternary Hermitian self-dual codes of lengths 26, 32, 36, 38 and 40.
\end{abstract}

\maketitle
\section{Introduction}
Hermitian self-dual codes form a class of linear codes which are self-dual with respect to the Hermitian inner product. Let $d(n)$ be the minimum distance of a quaternary Hermitian self-dual code of length $n$. An upper bound on $d(n)$ was given in \cite{1} as
	\begin{equation*}
	d(n)\leq 2\floor{n/6}+2.
	\end{equation*}

A Hermitian self-dual code whose minimum distance meets its corresponding bound is called \textit{extremal}. A Hermitian self-dual code with the highest possible minimum distance for its length is said to be \textit{optimal}. Extremal codes are necessarily optimal but optimal codes are not necessarily extremal. A \textit{best known} Hermitian self-dual code is a Hermitian self-dual code with the highest known minimum distance for its length. 

The existence of an extremal quaternary Hermitian self-dual code for lengths greater than 30 is still an open problem. It was proved in \cite{14} that there exists no extremal quaternary Hermitian self-dual code of length 26. A complete classification of quaternary Hermitian self-dual codes for lengths up to 22 is given in \cite{15}. Optimal quaternary Hermitian self-dual codes of length 24 possessing non-trivial automorphisms of order $\geq 3$ are classified in \cite{18,16}. Extremal quaternary Hermitian self-dual codes of length 28 possessing non-trivial odd order automorphisms are classified in \cite{19,20}. In \cite{11}, optimal and best known quaternary Hermitian self-dual of codes of lengths 24, 26, 32 and 34 are constructed. A classification of double circulant quaternary Hermitian self-dual codes for lengths up to 26 is given in \cite{12} and this classification is extended to lengths up to 40 in \cite{13}. In \cite{22}, it was proved that up to equivalence, there exists precisely one extremal quaternary Hermitian self-dual code of length 30 possessing a non-trivial automorphism of odd prime order.

The main purpose of this work is to use circulant matrices in order to construct Hermitian self-dual codes which are best known in the literature. A circulant matrix is a special type of Toeplitz matrix which is completely determined by a single vector. When searching for $n\times n$ matrices to use in constructions of (Hermitian) self-dual codes, by assuming these matrices are circulant we reduce the size of the search field from $n^2$ to $n$. For this reason, circulant matrices have been used extensively to construct (Hermitian) self-dual codes. See \cite{23,24,25,26,27,28,29} for recent utilisation of circulant matrices in constructing self-dual codes. In this work, we give three different modifications of various well-known circulant constructions of self-dual codes, which we apply to construct optimal and best known quaternary Hermitian self-dual codes. All of the new techniques can be used to construct Hermitian self-dual codes over any commutative Frobenius ring $R$ with a fixed non-trivial involutory automorphism $\hip{\phantom{a}}:R\to R$. We introduce these techniques and provide the conditions needed to produce a Hermitian self-dual code.

For the proofs of these techniques, we utilise a specialised mapping $\Theta$ which was used in \cite{4,17}. This mapping is inherently associated with the matrix product $B\hip{A}{}^T$, where $A$ and $B$ are $\lambda$-circulant matrices over $R$ such that $\lambda\hip{\lambda}=1$. If $A$ is the $\lambda$-circulant matrix generated by $\vct{a}\in R^n$, then using $\Theta$ allows us to verify the equality $A\hip{A}{}^T=-I_n$ by computing the values of $\floor{n/2}+1$ quantities in terms of $\vct{a}$. This eliminates the need to construct $A$ from its generating vector as well as computing the matrix product $A\hip{A}{}^T$ itself, which improves computational efficiency. We give and prove our own results concerning $\Theta$ as done so in \cite{4,17} and we further generalise them with respect to the Hermitian inner product.

Using the new techniques together with the building-up construction, we find many Hermitian self-dual codes with weight enumerator parameters of previously unknown values (relative to referenced sources). In total, 408 new codes are found, including 
	\begin{enumerate}[label=$\bullet$]
	\item 71 quaternary Hermitian self-dual $[26,13,8]$-codes;
	\item 82 quaternary Hermitian self-dual $[32,16,10]$-codes;
	\item 2 quaternary Hermitian self-dual $[36,18,12]$-codes;
	\item 1 quaternary Hermitian self-dual $[38,19,12]$-code;
	\item 252 quaternary Hermitian self-dual $[40,20,12]$-codes.
	\end{enumerate}

The rest of the work is organised as follows. In Section 2, we give preliminary definitions and results on Hermitian self-dual codes, Gray maps, circulant matrices, the specialised mapping $\Theta$ and the alphabets which we use. In Section 3, we introduce the new techniques and conditions needed to produce a Hermitian self-dual code. In Section 4, we apply the new techniques and the building-up construction to obtain the new Hermitian self-dual codes, whose weight enumerator parameter values we detail. We also tabulate the results in this section.

\section{Preliminaries}

\subsection{Hermitian Self-Dual Codes }

Let $R$ be a commutative Frobenius ring with a fixed non-trivial involutory automorphism $\hip{\phantom{a}}:R\to R$ (see \cite{9} for a full description of Frobenius rings and codes over Frobenius rings). Throughout this work, we always assume $R$ has unity. A code $\mathcal{C}$ of length $n$ over $R$ is a subset of $R^n$ whose elements are called codewords. If $\mathcal{C}$ is a submodule of $R^n$, then we say that $\mathcal{C}$ is linear. Let $\vct{x},\vct{y}\in R^n$ where $\vct{x}=(x_1,x_2,\ldots,x_n)$ and $\vct{y}=(y_1,y_2,\ldots,y_n)$. The Hermitian dual $\mathcal{C}^{\bot_H}$ of $\mathcal{C}$ is given by
	\begin{equation*}
	\mathcal{C}^{\bot_H}=\{\mathbf{x}\in R^n: \langle\mathbf{x},\mathbf{y}\rangle_H=0,\forall\mathbf{y}\in \mathcal{C}\},
	\end{equation*}	
where $\langle,\rangle_H$ denotes the Hermitian inner product defined by
	\begin{equation*}
	\langle\mathbf{x},\mathbf{y}\rangle_H=\sum_{i=1}^nx_i\hip{y_i}.
	\end{equation*}

We say that $\mathcal{C}$ is Hermitian self-orthogonal if $\mathcal{C}\subseteq \mathcal{C}^{\bot_H}$ and Hermitian self-dual if $\mathcal{C}=\mathcal{C}^{\bot_H}$.

If $R=\FF_{p^{2m}}$ for some prime $p$ and $m\in\NN$, then we define the involutory automorphism $\hip{a}=a^{p^m}$, $\forall a\in R$. We can extend this to the finite commutative ring $R=\FF_{p^{2m}}+u\FF_{p^{2m}}$ where
	\begin{equation*}
	\FF_{p^{2m}}+u\FF_{p^{2m}}=\{a+bu:a,b\in\FF_{p^{2m}},u^2=0\}
	\end{equation*}
and we define $\hip{a+bu}=a^{p^m}+b^{p^m}u$, $\forall a,b\in R$.

Two codes $\mathcal{C}$ and $\mathcal{C}'$ over $R$ are said to be \textit{conjugation equivalent} or simply \textit{equivalent} if there exists a monomial matrix $M$ over $R$ and an automorphism $\nu$ of $R$ such that $\mathcal{C}'=\nu(\mathcal{C}M)=\{\nu(\vct{c}M):\vct{c}\in \mathcal{C}\}$. If $\mathcal{C}=\nu(\mathcal{C}M)$, then $M$ and $\nu$ are said to form an \textit{automorphism} of $\mathcal{C}$. The set of all automorphisms of $\mathcal{C}$ forms the \textit{automorphism group} $\aut{\mathcal{C}}$ of $\mathcal{C}$.

\subsection{Alphabets}

In this paper, we consider the alphabets $\FF_{4}$ and $\FF_{4}+u\FF_{4}$.

We define $\FF_4\cong\FF_2[\omega]/\langle \omega^2+\omega+1\rangle$ so that
	\begin{equation*}
	\FF_4=\{a{\omega}+b(1+\omega): a,b\in\FF_2,\omega^2+\omega+1=0\}.
	\end{equation*}

Define
	\begin{equation*}
	\FF_4+u\FF_4=\{a+bu: a,b\in\FF_4,u^2=0\}.
	\end{equation*}

Then $\FF_4+u\FF_4$ is a commutative ring of order 16 and characteristic 2 such that $\FF_4+u\FF_4\cong\FF_4[u]/\langle u^2\rangle\cong\FF_2[\omega,u]/\langle \omega^2+\omega+1,u^2,\omega u+u\omega\rangle$.

We recall the following Gray map from \cite{2}:
	\begin{align*}
	\varphi_{\FF_4+u\FF_4}&:(\FF_4+u\FF_4)^n\to\FF_4^{2n},\\
		&\quad a+bu\mapsto(b,a+b),\,a,b\in\FF_4^n.
	\end{align*}

It was shown in \cite{3} that if $\mathcal{C}$ is a Hermitian self-dual code over $\FF_4+u\FF_4$ of length $n$, then $\varphi_{\FF_4+u\FF_4}(\mathcal{C})$ is a Hermitian self-dual code over $\FF_4$ of length $2n$. The Lee weight of a vector $\vct{x}\in(\FF_4+u\FF_4)^n$ can be defined to be $w_L(\vct{x})=n_1(\vct{x})+2n_2(\vct{x})$ where $n_1(\vct{x})$ is the number of components of $\vct{x}$ equal to $a+bu$ with $a=b$ or $b=0$ and $n_2(\vct{x})$ is the number of components of $\vct{x}$ equal to $a+bu$ with $a\neq b$ and $b\neq 0$. It is true that $\varphi_{\FF_4+u\FF_4}$ is an isometry from $(\FF_4+u\FF_4)^n$ under Lee distance to $\FF_4^{2n}$ under Hamming distance. In this way, the minimum Lee distance and Lee weight enumerator of a code $\mathcal{C}$ over $\FF_4+u\FF_4$ are equal to the minimum Hamming distance and Hamming weight enumerator of $\varphi_{\FF_4+u\FF_4}(\mathcal{C})$, respectively.
	
\subsection{Special Matrices}

We now recall the definitions and properties of some special matrices which we use in our work. Let $\vct{a}=(a_0,a_1,\ldots,a_{n-1})\in R^n$ where $R$ is a commutative ring and let
	\begin{equation*}
	A=\begin{pmatrix}
	a_0 & a_1 & a_2 & \cdots & a_{n-1}\\
	\lambda a_{n-1} & a_0 & a_1 & \cdots & a_{n-2}\\
	\lambda a_{n-2} & \lambda a_{n-1} & a_0 & \cdots & a_{n-3}\\
	\vdots & \vdots & \vdots & \ddots & \vdots\\
	\lambda a_1 & \lambda a_2 & \lambda a_3 & \cdots & a_0
	\end{pmatrix},
	\end{equation*}
where $\lambda\in R$. Then $A$ is called the $\lambda$-circulant matrix generated $\vct{a}$, denoted by $A=\pcir{\lambda}{\vct{a}}$. If $\lambda=1$, then $A$ is called the circulant matrix generated by $\vct{a}$ and is more simply denoted by $A=\cir{\vct{a}}$. If we define the matrix
	\begin{equation*}
	P_{\lambda}=\begin{pmatrix}
	\vct{0} & I_{n-1}\\
	\lambda & \vct{0}
	\end{pmatrix},
	\end{equation*}
then it follows that $A=\sum_{i=0}^{n-1}a_iP_{\lambda}^i$. Clearly, the sum of any two $\lambda$-circulant matrices is also a $\lambda$-circulant matrix. If $B=\pcir{\lambda}{\vct{b}}$ where $\vct{b}=(b_0,b_1,\ldots,b_{n-1})\in R^n$, then $AB=\sum_{i=0}^{n-1}\sum_{j=0}^{n-1}a_ib_jP_{\lambda}^{i+j}$. Since $P_{\lambda}^n=\lambda I_n$ there exist $c_k\in R$ such that $AB=\sum_{k=0}^{n-1}c_kP_{\lambda}^k$ so that $AB$ is also $\lambda$-circulant. In fact, it is true that
	\begin{equation*}
	c_{k}=\sum_{\substack{[i+j]_n=k\\i+j<n}}a_ib_j+\sum_{\substack{[i+j]_n=k\\i+j\geq n}}\lambda a_ib_j=\vct{x}_1\vct{y}_{k+1}
	\end{equation*}
for $k\in[0\zint n-1]$, where $\vct{x}_i$ and $\vct{y}_i$ respectively denote the $i\nth$ row and column of $A$ and $B$ and $[i+j]_n$ denotes the smallest non-negative integer such that $i\equiv j\mod{n}$. From this, we can see that $\lambda$-circulant matrices commute multiplicatively. Moreover, if $\lambda$ is a unit in $R$, then $A^T$ is $\lambda^{-1}$-circulant such that $A^T=a_0I_n+\lambda\sum_{i=1}^{n-1}a_{n-i}P_{\lambda^{-1}}^i$. It follows then that $AA^T$ is $\lambda$-circulant if and only if $\lambda$ is involutory in $R$, i.e. $\lambda^2=1$.

Let $A_i=\pcir{\mu}{\vct{a}_i}$ where $\vct{a}_i\in R^n$ for $i\in[0\zint n-1]$. Then we can interpret $\pcir{\lambda}{(A_0,A_1,\ldots,A_{n-1})}$ as a $\lambda$-circulant matrix with respect to its $\mu$-circulant blocks $A_i$ and we call such a matrix a block $(\lambda,\mu)$-circulant matrix.

Let $X$ be an $m\times n$ matrix over a commutative ring $R$ with a fixed non-trivial involutory automorphism $\hip{\phantom{a}}:R\to R$ whose $(i,j)\nth$ entry is $x_{i,j}$. Then $\hip{X}$ denotes the $m\times n$ matrix whose $(i,j)\nth$ entry is $\hip{x_{i,j}}$. Since $\hip{\phantom{a}}$ is an involutory automorphism, it obeys equalities such as $\hip{I_n}=I_n$, $\hip{\hip{X}}=X$, $\hip{-X}=-\hip{X}$ and $\hip{X}\,\hip{Y}=\hip{XY}$ for any $n\times m$ matrix $Y$ over $R$. If $A$ is $\lambda$-circulant such that $A=\sum_{i=0}^{n-1}a_iP_{\lambda}^i$, then $\hip{A}=\sum_{i=0}^{n-1}\hip{a_i}\,\hip{P_{\lambda}^i}=\sum_{i=0}^{n-1}\hip{a_i}P_{\hip{\lambda}}^i$ so that $\hip{A}$ is $\hip{\lambda}$-circulant. Moreover, if $\hip{\lambda}$ is a unit in $R$, then $\hip{A}{}^T$ is $\hip{\lambda}{}^{-1}$-circulant such that $\hip{A}{}^T=\hip{a_0}I_n+\hip{\lambda}\sum_{i=1}^{n-1}\hip{a_{n-i}}P_{\hip{\lambda}{}^{-1}}^i$. It follows then that $A\hip{A}{}^T$ is $\lambda$-circulant if and only if $\lambda\hip{\lambda}=1$.

Let $J_n$ be an $n\times n$ matrix over $R$ whose $(i,j)\nth$ entry is $1$ if $i+j=n+1$ and 0 if otherwise. Then $J_n$ is called the $n\times n$ exchange matrix and corresponds to the row-reversed (or column-reversed) version of $I_n$. We see that $J_n$ is both symmetric and involutory, i.e. $J_n=J_n^T$ and $J_n^2=I_n$. For any matrix $A\in R^{m\times n}$, premultiplying $A$ by $J_m$ and postmultiplying $A$ by $J_n$ inverts the order in which the rows and columns of $A$ appear, respectively. Namely, the $(i,j)\nth$ entries of $J_mA$ and $AJ_n$ are the $([1-i]_m,j)\nth$ and $(i,[1-j]_n)\nth$ entries of $A$, respectively. Note that $[i+j]_n$ corresponds to the $(i,j)\nth$ entry of the matrix $J_nV$ where $V=\cir{(0,1,2,\ldots,n-1)}$.

\subsection{A Special Mapping}
We now introduce and explore the properties of a mapping which was used in \cite{4,17}. The mapping is inherently associated with the matrix product $B\hip{A}{}^T$, where $A$ and $B$ are $\lambda$-circulant matrices such that $\lambda\hip{\lambda}=1$. By utilising $\Theta$, we are able to improve the computational efficiency of our algorithms.
	\begin{dfn}\textup{(\cite{4,17})}
	Let $R$ be a commutative ring and let $n\in\NN$ be fixed. Let $\Theta:R^n\times R^n\times\ZZ_n\to R$ be a mapping with an optional argument $\lambda\in R$ defined by
		\begin{equation*}
		\Theta(\vct{x},\vct{y},j)[\lambda]=\sum_{i=0}^{n-j-1}x_{[i+j]_n}y_i+\lambda\sum_{i=n-j}^{n-1}x_{[i+j]_n}y_i,
		\end{equation*}		
	where $\vct{x}=(x_0,x_1,\ldots,x_{n-1}),\vct{y}=(y_0,y_1,\ldots,y_{n-1})\in R^n$ and $j\in[0\zint n-1]$.
		
	If $j=0$, we define
		\begin{equation*}
		\Theta(\vct{x},\vct{y},0)=\sum_{i=0}^{n-1}x_iy_i=\vct{x}\vct{y}{}^T,
		\end{equation*}
	which is independent of $\lambda$.
		
	If $\lambda$ is unspecified, then we assume $\lambda=1$ so that
		\begin{equation*}
		\Theta(\vct{x},\vct{y},j)=\sum_{i=0}^{n-1}x_{[i+j]_n}y_i.
		\end{equation*}
	\end{dfn}

	\begin{lemm}\label{lemm-1}\textup{(\cite{4})}
	Let $R$ be a commutative ring with a fixed non-trivial involutory automorphism $\hip{\phantom{a}}:R\to R$. Let $\vct{x},\vct{y}\in R^n$ and let $\lambda\in R:\lambda\hip{\lambda}=1$. Then $\Theta(\vct{x},\vct{y},j)[\lambda]=\lambda\Theta(\vct{y},\vct{x},n-j)[\hip{\lambda}]$, $\forall j\in[0\zint n-1]$.
	\end{lemm}

	\begin{proof}
	If $\vct{x}=(x_0,x_1,\ldots,x_{n-1})$, then $x_{[i+k]_n}=\tilde{x}_s$, where $\tilde{\vct{x}}$ is the vector $\vct{x}$ after being circularly shifted by $k$ places for some $k\in[0\zint n-1]$. If $\vct{y}=(y_0,y_1,\ldots,y_{n-1})$, then in a correspondence between the elements $x_i$ and $y_i$, inflicting a circular shift to both $\vct{x}$ and $\vct{y}$ by the same number of places preserves this correspondence. Thus, noting that $\lambda\hip{\lambda}=1$ by assumption, we have
	\begin{align*}
		\lambda\Theta(\vct{y},\vct{x},n-j)[\hip{\lambda}]&=\lambda\left(\sum_{i=0}^{n-(n-j)-1}y_{[i+(n-j)]_n}x_i+\hip{\lambda}\sum_{i=n-(n-j)}^{n-1}y_{[i+(n-j)]_n}x_i\right)\\
		&=\lambda\sum_{i=0}^{j-1}y_{[i+(n-j)]_n}x_i+\sum_{i=j}^{n-1}y_{[i+(n-j)]_n}x_i\\
		&=\sum_{i=j}^{n-1}x_{[i+j+(n-j)]_n}y_{[i+(n-j)]_n}+\lambda\sum_{i=0}^{j-1}x_{[i+j+(n-j)]_n}y_{[i+(n-j)]_n}\\
		&=\sum_{i=0}^{n-j-1}x_{[i+j]_n}y_i+\lambda\sum_{i=n-j}^{n-1}x_{[i+j]_n}y_i\\
		&=\Theta(\vct{x},\vct{y},j)[\lambda].
	\end{align*}
	\end{proof}

	\begin{remark}\label{remark-1}
	In Lemma \ref{lemm-1}, suppose we want to calculate $f(j)=\Theta(\vct{x},\hip{\vct{x}},j)[\hip{\lambda}]$, $\forall j\in[0\zint n-1]$. We have $f(j)=\lambda\Theta(\vct{x},\hip{\vct{x}},n-j)[\hip{\lambda}]$ which, since $\lambda\hip{\lambda}=1$, implies
		\begin{equation*}
		\hip{f(n-j)}=\Theta(\hip{\vct{x}},\vct{x},n-j)[\lambda]=\lambda\Theta(\vct{x},\hip{\vct{x}},j)[\hip{\lambda}]=\lambda f(j)
		\end{equation*}
	so that $f(j)=\lambda^{-1}\hip{f(n-j)}=\hip{\lambda f(n-j)}$. Therefore, to calculate $f(j)$ for $j\in[0\zint n-1]$, it is sufficient to determine $f(j)$ for $j\in[0\zint\floor{n/2}]$.

	Likewise, let $\vct{a}_i\in R^n$ for $i\in[0\zint k-1]$ and suppose we want to calculate $g(j,i,t)=\Theta(\vct{a}_{[i+j]_k},\hip{\vct{a}_i},t)[\hip{\lambda}]$, $\forall j\in[1\zint k-1]$ and $t\in[1\zint n-1]$. By Lemma \ref{lemm-1} we have 
		\begin{align*}
		g(j,i,t)&=\Theta(\vct{a}_{[i+j]_k},\hip{\vct{a}_i},t)[\hip{\lambda}]\\
		&=\hip{\lambda}\Theta(\hip{\vct{a}_i},\vct{a}_{[i+j]_k},n-t)[\lambda]\\
		&=\hip{\lambda\Theta(\vct{a}_i,\hip{\vct{a}_{[i+j]_k}},n-t)[\hip{\lambda}]}\\
		&=\hip{\lambda\Theta(\vct{a}_{[(i+j)+(k-j)]_k},\hip{\vct{a}_{[i+j]_k}},n-t)[\hip{\lambda}]}\\
		&=\hip{\lambda g([i+j]_k,k-j,n-t)}.
		\end{align*}
	
	Let $G_t$ be the matrix whose $(i,j)\nth$ entry is $g(j,i,t)$ for fixed $t\in[1\zint n-1]$. Let $T_1$ be the transformation which multiplies each entry of a matrix by $\lambda$ and let $T_2$ be the transformation which circularly shifts the $j\nth$ row of a matrix to the right by $j$ places, $\forall j$. Then we see that $G_t=\hip{T_2(J_{k-1}T_1(G_{n-t}))}$, $\forall t$, where $J_{k-1}$ is the $(k-1)\times(k-1)$ exchange matrix. Conversely, since $T_1$ and $T_2$ are clearly both invertible, we have $G_{n-t}=T_1^{-1}(J_{k-1}T_2^{-1}(\hip{G_t}))$. Therefore, to calculate $g(j,i,t)$ for $j\in[1\zint k-1]$ and $t\in[1\zint n-1]$, it is sufficient to determine $g(j,i,t)$ for $j\in[1\zint k-1]$ and $t\in[1\zint\floor{n/2}]$.
	 
	Finally, suppose we want to calculate $g(j,i,0)=\Theta(\vct{a}_{[i+j]_k},\hip{\vct{a}_i},0)$, $\forall j\in[1\zint k-1]$. Following a similar argument, we have $g(j,i,0)=\hip{g([i+j]_k,k-j,0)}$. Let $\vct{v}_j$ be the vector whose $i\nth$ entry is $g(j,i,0)$. Then we see that $\vct{v}_j$ corresponds to the vector $\hip{\vct{v}_{k-j}}$ after be circularly shifted to the left by $j$ places. Therefore, to calculate $g(j,i,0)$ for $j\in[1\zint k-1]$, it is sufficient to determine $g(j,i,0)$ for $j\in[1\zint\floor{k/2}]$.
	\end{remark}

	\begin{lemm}\label{lemm-2}\textup{(\cite{4})}
	Let $R$ be a commutative ring with a fixed non-trivial involutory automorphism $\hip{\phantom{a}}:R\to R$. Let $A=\pcir{\lambda}{\vct{a}}$ and $B=\pcir{\lambda}{\vct{b}}$ with $\vct{a},\vct{b}\in R^n$ and $\lambda\in R:\lambda\hip{\lambda}=1$. Then $B\hip{A}{}^T=\pcir{\lambda}{(v_0,v_1,\ldots,v_{n-1})}$, where $v_j=\Theta(\vct{b},\hip{\vct{a}},j)[\hip{\lambda}]$, $\forall j\in[0\zint n-1]$.
	\end{lemm}

	\begin{proof}
	Since $\lambda\hip{\lambda}=1$ by assumption, we know that $B\hip{A}{}^T$ is $\lambda$-circulant such that $B\hip{A}{}^T=\pcir{\lambda}{(\vct{x}_1\hip{\vct{y}_1}{}^T,\vct{x}_1\hip{\vct{y}_2}{}^T,\ldots,\vct{x}_1\hip{\vct{y}_n}{}^T)}$, where $\vct{x}_i$ and $\vct{y}_i$ denote the $i\nth$ rows of $B$ and $A$, respectively. Let $B\hip{A}{}^T=\pcir{\lambda}{(v_0,v_1,\ldots,v_{n-1})}$ so that $v_j=\vct{x}_1\hip{\vct{y}_{j+1}}{}^T$ for $j\in[0\zint n-1]$. It is easy to see that $v_0=\sum_{i=0}^{n-1}b_i\hip{a_i}$. In the product $v_1=\vct{x}_1\hip{\vct{y}_2}{}^T$, we see that the indices of the vector $\hip{\vct{y}_2}=(\hip{\lambda a_{n-1}},\hip{a_0},\hip{a_1},\ldots,\hip{a_{n-2}})$ correspond to the indices of the vector $\vct{x}_1=(b_0,b_1,b_2,\ldots,b_{n-1})$ after being circularly shifted to the right by 1 place. Thus, in $v_1=\vct{x}_1\hip{\vct{y}_2}{}^T$, there is a summation of terms in the form $b_{[i+1]_n}\hip{a_i}$ for $i\in[0\zint n-1]$. By extending this argument, we see that in the product $\vct{x}_1\hip{\vct{y}_{j+1}}{}^T$, there is a summation of terms in the form $b_{[i+j]_n}\hip{a_i}$ for $i\in[0\zint n-1]$ and $j\in[1\zint n-1]$. Also, in $v_1$, we see that the terms of the summation will acquire $\hip{\lambda}$ as a coefficient for $i=n-1$. By extending this argument, in $v_j$, we see that the terms of the summation will acquire $\hip{\lambda}$ as a coefficient for $i\in[n-j\zint n-1]$ and $j\in[1\zint n-1]$. In summary, we have
		\begin{equation*}
		v_j=\begin{cases}
		\sum_{i=0}^{n-1}b_i\hip{a_i},& j=0,\\
		\sum_{i=0}^{n-j-1}b_{[i+j]_n}\hip{a_i}+\hip{\lambda}\sum_{i=n-j}^{n-1}b_{[i+j]_n}\hip{a_i},& j\in[1\zint n-1]
		\end{cases}
		\end{equation*}	
	and so, in terms of $\Theta$, we see that $v_j=\Theta(\vct{b},\hip{\vct{a}},j)[\hip{\lambda}]$, $\forall j\in[0\zint n-1]$.
	\end{proof}

With these lemmas established, we will now look at how $\Theta$ can be used to prove that a matrix $G=(I_n,A)$ is a generator matrix of a Hermitian self-dual code.
	\begin{prop}\label{prop-3}\textup{(\cite{4})}
	Let $R$ be a commutative ring with a fixed non-trivial involutory automorphism $\hip{\phantom{a}}:R\to R$. Let $A=\pcir{\lambda}{\vct{a}}$ with $\vct{a}\in R^n$ and $\lambda\in R:\lambda\hip{\lambda}=1$. Then $A\hip{A}{}^T=-I_n$ if and only if
		\begin{equation*}
		\Theta(\vct{a},\hip{\vct{a}},j)[\hip{\lambda}]=\begin{cases}
		-1,& j=0,\\
		0,& j\in[1\zint\floor{n/2}].
		\end{cases}
		\end{equation*}	
	\end{prop}

	\begin{proof}
	Since $\lambda\hip{\lambda}=1$ by assumption, by Lemma \ref{lemm-2} we have $A\hip{A}{}^T=\pcir{\lambda}{(v_0,v_1,\ldots,v_{n-1})}$ where $v_j=\Theta(\vct{a},\hip{\vct{a}},j)[\hip{\lambda}]$, $\forall j\in[0\zint n-1]$. The main diagonal and off-diagonal entries of $A\hip{A}{}^T$ are given by $v_0$ and $v_j$, respectively, for $j\neq 0$, so $A\hip{A}{}^T=-I_n$ if and only if $v_0=-1$ and $v_j=0$, $\forall j\in[1\zint n-1]$. By Remark \ref{remark-1}, we see that it is sufficient to verify $v_0=-1$ and $v_j=0$, $\forall j\in[1\zint\floor{n/2}]$. Therefore, we see that $A\hip{A}{}^T=-I_n$ if and only if
		\begin{equation*}
		\Theta(\vct{a},\hip{\vct{a}},j)[\hip{\lambda}]=\begin{cases}
		-1,& j=0,\\
		0,& j\in[1\zint\floor{n/2}].
		\end{cases}
		\end{equation*}	
	\end{proof}

For example, if $R$ is a commutative Frobenius ring with a fixed non-trivial involutory automorphism $\hip{\phantom{a}}:R\to R$ such that $\lambda\in R:\lambda\hip{\lambda}=1$, then in the pure double circulant construction of Hermitian self-dual codes given by $G=(I_n,A)$ for an $n\times n$ $\lambda$-circulant matrix $A=\pcir{\lambda}{\vct{a}}$ over $R$, we know that $G$ is a generator matrix of a Hermitian self-dual $[2n,n]$-code over $R$ if and only if $A\hip{A}{}^T=-I_n$. In terms of $\Theta$, by Proposition \ref{prop-3} this is true if and only if
	\begin{equation*}
	\Theta(\vct{a},\hip{\vct{a}},j)[\hip{\lambda}]=\begin{cases}
	-1,& j=0,\\
	0,& j\in[1\zint\floor{n/2}].
	\end{cases}
	\end{equation*}	

\section{The Constructions}

In this section, we present the three techniques for constructing Hermitian self-dual codes, all of which are derived as modifications of previously known constructions of self-dual codes. We will hereafter always assume $R$ is a commutative Frobenius ring with a fixed non-trivial involutory automorphism $\hip{\phantom{a}}:R\to R$, which we refer to as the Hermitian involution.

\subsection{Construction 1}

The first technique we look at can be used to construct Hermitian self-dual $[4n,2n]$-codes over $R$. It can be interpreted as a Hermitian modification of the four circulant technique for constructing self-dual codes first introduced in \cite{5}, which uses a matrix $G$ defined by
	\begin{equation*}
	G=
	\begin{pmatrix}
	I_{2n} & X
	\end{pmatrix},\quad\text{where }
	X=
	\begin{pmatrix}
	A & B\\-B^T & A^T
	\end{pmatrix},
	\end{equation*}
and where $A$ and $B$ are circulant matrices. It also corresponds to the Hermitian analogue of the construction presented in \cite{17}.

	\begin{thm}\label{thm-1}
	Let
		\begin{equation*}
		G=\begin{pmatrix}
		I_{2n} & X
		\end{pmatrix},\quad\text{where }
		X=\begin{pmatrix}
		-A^TCJ & -\hip{B}\\
		B^TCJ & -\hip{A}
		\end{pmatrix}
		\end{equation*}	
	and where $J=J_n$, $A=\pcir{\lambda}{\vct{a}}$, $B=\pcir{\lambda}{\vct{b}}$ and $C=\pcir{\mu}{\vct{c}}$ with $\vct{a},\vct{b},\vct{c}\in R^n$ and $\lambda,\mu\in R:\lambda\hip{\lambda}=\mu\hip{\mu}=1$. Then $G$ is a generator matrix of a Hermitian self-dual $[4n,2n]$-code over $R$ if and only if
		\begin{empheq}[left=\empheqlbrace]{align*}
		\sum_{\vct{x}\in S}\Theta(\vct{x},\hip{\vct{x}},j)[\hip{\lambda}]&=\begin{cases}
		-1,& j=0,\\
		0,& j\in[1\zint\floor{n/2}],
		\end{cases}\\
		\Theta(\vct{c},\hip{\vct{c}},j)[\hip{\mu}]&=\begin{cases}
		1,& j=0,\\
		0,& j\in[1\zint\floor{n/2}],
		\end{cases}
		\end{empheq}		
	where $S=\{\vct{a},\vct{b}\}$.	
	\end{thm}
	\begin{proof}
	We know that $G$ is a generator matrix of a Hermitian self-dual $[4n,2n]$-code over $R$ if and only if $X\hip{X}{}^T=-I_{2n}$. Since $\lambda\hip{\lambda}=1$ by assumption, we have that $A$ and $B$ as well as their Hermitian involutions and transpositions all commute with one another multiplicatively. Firstly, since $J$ is symmetric we have
		\begin{equation*}
		\hip{X}{}^T=\begin{pmatrix}
		-\hip{A}{}^T\hip{C}J & -B\\
		\hip{B}{}^T\hip{C}J & -A
		\end{pmatrix}^T=
		\begin{pmatrix}
		-J\hip{C}{}^T\hip{A} & J\hip{C}{}^T\hip{B}\\
		-B^T & -A^T
		\end{pmatrix}.
		\end{equation*}	
		
	If the $(i,j)\nth$ block-wise entry of $X\hip{X}{}^T$ is $x_{i,j}$, noting that $J$ is involutory we see that
		\begin{align*}
		x_{1,1}&=A^TCJ^2\hip{C}{}^T\hip{A}+\hip{B}B^T=A^TC\hip{C}{}^T\hip{A}+\hip{B}B^T,\\
		x_{1,2}&=-A^TCJ^2\hip{C}{}^T\hip{B}+\hip{B}A^T=-A^TC\hip{C}{}^T\hip{B}+\hip{B}A^T,\\
		x_{2,1}&=-B^TCJ^2\hip{C}{}^T\hip{A}+\hip{A}B^T=-B^TC\hip{C}{}^T\hip{A}+\hip{A}B^T,\\
		x_{2,2}&=B^TCJ^2\hip{C}{}^T\hip{B}{}^T+\hip{A}A^T=B^TC\hip{C}{}^T\hip{B}{}^T+\hip{A}A^T.
		\end{align*}	
		
	Noting that $X\hip{X}{}^T=-I_{2n}$ if and only if $A^TC\hip{C}{}^T\hip{B}=\hip{B}A^T$, we see that
		\begin{align*}
		X\hip{X}{}^T=-I_{2n}&\iff A^TC\hip{C}{}^T\hip{A}+\hip{B}B^T=-I_n\\&\iff A^TC\hip{C}{}^T\hip{A}A^T+(\hip{B}A^T)B^T=-A^T\\&\iff A^TC\hip{C}{}^T\hip{A}A^T+(A^TC\hip{C}{}^T\hip{B})B^T=-A^T\\&\iff A^TC\hip{C}{}^T(\hip{A}A^T+\hip{B}B^T)=-A^T\\&\iff
		C\hip{C}{}^T(\hip{A}A^T+\hip{B}B^T)=-I_n
		\end{align*}
	and so, combined with our other required conditions, we have that $X\hip{X}{}^T=-I_{2n}$ if and only if
		\begin{empheq}[left=\empheqlbrace]{align*}
		A^TC\hip{C}{}^T\hip{A}+\hip{B}B^T&=-I_n,\\
		B^TC\hip{C}{}^T\hip{B}+\hip{A}A^T&=-I_n,\\
		C\hip{C}{}^T(\hip{A}A^T+\hip{B}B^T)&=-I_n.
		\end{empheq}	
		
	Clearly, we must have $C\hip{C}{}^T=I_n$ for all of these equations to be satisfied. With this prerequisite, our conditions reduce to
		\begin{align*}
		x_{1,1}&=\hip{A}A^T+\hip{B}B^T,\\
		x_{1,2}&=\vct{0},\\
		x_{2,1}&=\vct{0},\\
		x_{2,2}&=\hip{A}A^T+\hip{B}B^T,
		\end{align*}
	or equivalently
		\begin{align*}
		x_{1,1}&=A\hip{A}{}^T+B\hip{B}{}^T,\\
		x_{1,2}&=\vct{0},\\
		x_{2,1}&=\vct{0},\\
		x_{2,2}&=A\hip{A}{}^T+B\hip{B}{}^T,
		\end{align*}
		
	Therefore, $X\hip{X}{}^T=-I_{2n}$ if and only if $A\hip{A}{}^T+B\hip{B}{}^T=-I_n$ and $C\hip{C}{}^T=I_n$ and by Proposition \ref{prop-3} this is true if and only if
		\begin{empheq}[left=\empheqlbrace]{align*}
		\sum_{\vct{x}\in S}\Theta(\vct{x},\hip{\vct{x}},j)[\hip{\lambda}]&=\begin{cases}
		-1,& j=0,\\
		0,& j\in[1\zint\floor{n/2}],
		\end{cases}\\
		\Theta(\vct{c},\hip{\vct{c}},j)[\hip{\mu}]&=\begin{cases}
		1,& j=0,\\
		0,& j\in[1\zint\floor{n/2}],
		\end{cases}
		\end{empheq}		
	where $S=\{\vct{a},\vct{b}\}$.
	\end{proof}
	
	\begin{remark}
		Let $U'$ denote the set of unitary elements in $R$, i.e. $U'=\{\lambda\in R:\lambda\hip{\lambda}=1\}$. Let $N_{C}=N_C(R,n)$ denote the number of unitary $\mu$-circulant matrices $C$ (i.e. $C\hip{C}{}^T=I_n$) over $R$ for all $\mu\in U'$. The search field for Hermitian self-dual $[4n,2n]$-codes over $R$ constructed by Theorem \ref{thm-1} is of size $|R|^{2n}\cdot |U'|\cdot N_C$. In general, $N_C$ is relatively small, for example $N_C(\FF_4,10)=4,320$ and $N_C(\FF_4+u\FF_4,5)=8,640$.
	\end{remark}
	\begin{remark}\label{remark-2}
	In Theorem \ref{thm-1}, we are in fact able to assume $C$ is any matrix over $R$ such that $C$ is unitary. Moreover, $C$ and $\hip{C}{}^T$ need not commute multiplicatively with either $A$, $B$, their Hermitian involutions or their transpositions.
	\end{remark}
	
\subsection{Construction 2}
The second technique can be used to construct Hermitian self-dual $[2kn,kn]$-codes over $R$ where $k\in\NN$. It is derived as the Hermitian analogue of the technique for constructing self-dual codes first given in \cite{6}, which uses a matrix $G$ defined by
	\begin{equation*}
		G=\begin{pmatrix}
		I_{kn} & X
		\end{pmatrix},\quad\text{where }
		X=\cir{(A_0,A_1,\ldots,A_{k-1})},
	\end{equation*}
and where $A_i$ are circulant matrices for $i\in[0\zint k-1]$. 

	\begin{thm}\label{thm-2}
	Let
		\begin{equation*}
		G=\begin{pmatrix}
		I_{kn} & X
		\end{pmatrix},\quad\text{where }
		X=\pcir{\lambda}{(A_0,A_1,\ldots,A_{k-1})}
		\end{equation*}	
	and where $A_i=\pcir{\mu}{\vct{a}_i}$ with $\vct{a}_i\in R^n$ for $i\in[0\zint k-1]$ and $\lambda,\mu\in R:\lambda\hip{\lambda}=\mu\hip{\mu}=1$. Then $G$ is a generator matrix of a Hermitian self-dual $[2kn,kn]$-code over $R$ if and only if
		\begin{empheq}[left=\empheqlbrace]{align*}
		\sum_{i=0}^{k-1}\Theta(\vct{a}_i,\hip{\vct{a}_i},t)[\hip{\mu}]&=\begin{cases}
		-1,&t=0\\
		0,&t\in[1\zint\floor{n/2}],
		\end{cases}\\
		\sum_{i=0}^{k-j-1}\Theta(\vct{a}_{[i+j]_k},\hip{\vct{a}_i},0)+\hip{\lambda}\sum_{i=k-j}^{k-1}\Theta(\vct{a}_{[i+j]_k},\hip{\vct{a}_i},0)&=0,\quad j\in[1\zint\floor{k/2}],\\
		\sum_{i=0}^{k-j-1}\Theta(\vct{a}_{[i+j]_k},\hip{\vct{a}_i},t)[\hip{\mu}]+\hip{\lambda}\sum_{i=k-j}^{k-1}\Theta(\vct{a}_{[i+j]_k},\hip{\vct{a}_i},t)[\hip{\mu}]&=0,\quad\begin{aligned}
		&j\in[1\zint k-1],\\
		&t\in[1\zint\floor{n/2}].
		\end{aligned}
		\end{empheq}		
	\end{thm}	

	\begin{proof}
	We know that $G$ is a generator matrix of a Hermitian self-dual $[2kn,kn]$-code over $R$ if and only if $X\hip{X}{}^T=-I_{kn}$, namely the main diagonal and off-diagonal block-wise entries of $X\hip{X}{}^T$ are equal to $-I_n$ and $\vct{0}$, respectively. Since $\lambda\hip{\lambda}=1$ by assumption, we have that $X\hip{X}{}^T$ is $\lambda$-circulant such that $X\hip{X}{}^T=\pcir{\lambda}{(\vct{R}_1\hip{\vct{R}_1}{}^T,\vct{R}_1\hip{\vct{R}_2}{}^T,\ldots,\vct{R}_1\hip{\vct{R}_n}{}^T)}$, where $\vct{R}_i$ denotes the $i\nth$ block row of $X$. Following an argument similar to that used in the proof of Lemma \ref{lemm-2}, we observe that
		\begin{equation*}
			\vct{R}_1\hip{\vct{R}_j}{}^T=\begin{cases}
			\sum_{i=0}^{k-1}A_i\hip{A_i}{}^T,&j=1,\\
			\sum_{i=0}^{k-1}A_{[i+(j-1)]_k}\hip{A_i}{}^T+\hip{\lambda}\sum_{i=k-j+1}^{k-1}A_{[i+(j-1)]_k}\hip{A_i}{}^T,&j\in[2\zint k].
			\end{cases}\geteq\label{eqn-1}
		\end{equation*}
	
	The main diagonal and off-diagonal block-wise entries of $X\hip{X}{}^T$ are equal to $\vct{R}_1\hip{\vct{R}_1}{}^T$ and $\vct{R}_1\hip{\vct{R}_j}{}^T$ for $j\in[2\zint k]$, respectively. Thus, by \eqref{eqn-1} we see that $X\hip{X}{}^T=-I_{kn}$ if and only if
		\begin{equation*}
			\sum_{i=0}^{k-1}A_i\hip{A_i}{}^T=-I_n\geteq\label{eqn-2}
		\end{equation*}
	and
		\begin{equation*}
			\sum_{i=0}^{k-j-1}A_{[i+j]_k}\hip{A_i}{}^T+\hip{\lambda}\sum_{i=k-j}^{k-1}A_{[i+j]_k}\hip{A_i}{}^T=\vct{0},\geteq\label{eqn-3}
		\end{equation*}
	$\forall j\in[1\zint k-1]$. Since $\mu\hip{\mu}=1$, by Proposition \ref{prop-3} we see that \eqref{eqn-2} is satisfied if and only if
		\begin{equation*}
			\sum_{i=0}^{k-1}\Theta(\vct{a}_i,\hip{\vct{a}_i},t)[\hip{\mu}]=\begin{cases}
			-1,&t=0,\\
			0,&t\in[1\zint\floor{n/2}].
			\end{cases}
		\end{equation*}
	
	By Lemma \ref{lemm-2}, we see that $A_{[i+j]_k}\hip{A_i}{}^T=\pcir{\lambda}{(v_{j,0},v_{j,1},\ldots,v_{j,n-1})}$ where $v_{j,t}=\Theta(\vct{a}_{[i+j]_k},\hip{\vct{a}_i},t)[\hip{\mu}]$, $\forall j\in[1\zint k-1]$ and $t\in[0\zint n-1]$. By Remark \ref{remark-1}, we know that it is sufficient to determine $v_{j,t}$, $\forall j\in[1\zint k-1]$ and $t\in[1\zint\floor{n/2}]$ and $v_{j,0}$ for $j\in[1\zint\floor{k/2}]$. Therefore, we find that \eqref{eqn-3} is satisfied if and only if
		\begin{equation*}
		\sum_{i=0}^{k-j-1}\Theta(\vct{a}_{[i+j]_k},\hip{\vct{a}_i},0)+\hip{\lambda}\sum_{i=k-j}^{k-1}\Theta(\vct{a}_{[i+j]_k},\hip{\vct{a}_i},0)=0,
		\end{equation*}	
	$\forall j\in[1\zint\floor{k/2}$ and
		\begin{equation*}
		\sum_{i=0}^{k-j-1}\Theta(\vct{a}_{[i+j]_k},\hip{\vct{a}_i},t)[\hip{\mu}]+\hip{\lambda}\sum_{i=k-j}^{k-1}\Theta(\vct{a}_{[i+j]_k},\hip{\vct{a}_i},t)[\hip{\mu}]=0,
		\end{equation*}	
	$\forall j\in[1\zint k-1]$ and $t\in[1\zint\floor{n/2}]$.		
	\end{proof}
	\begin{remark}
		Let $U'$ denote the set of unitary elements in $R$, i.e. $U'=\{\lambda\in R:\lambda\hip{\lambda}=1\}$. The search field for Hermitian self-dual $[2kn,kn]$-codes over $R$ constructed by Theorem \ref{thm-2} is of size $|R|^{kn}\cdot |U'|^2$.
	\end{remark}	
	
\subsection{Construction 3} 

The third technique can be used to construct Hermitian self-dual $[2(kn+1),kn+1]$-codes over $R$ where $k\in\NN$. It is derived as the Hermitian analogue of the technique for constructing self-dual codes first given in \cite{6}, which uses a matrix $G$ defined by
	\begin{equation*}
		G=\begin{pmatrix}
		I_{kn+1} & X
		\end{pmatrix},\quad\text{where }
		X=\begin{pmatrix}
		x_1 & X_2\\
		X_3^T & Y
		\end{pmatrix},\quad\text{and }
		Y=\cir{(A_0,A_1,\ldots,A_{k-1})},
	\end{equation*}
and where $A_i$ are circulant matrices for $i\in[0\zint k-1]$ with $X_2=(x_2,x_2,\ldots,x_2)$ and $X_3=(x_3,x_3,\ldots,x_3)$ for elements $x_1,x_2,x_3\in R$.

	\begin{thm}\label{thm-3}
	Let
		\begin{equation*}
		G=\begin{pmatrix}
		I_{kn+1} & X
		\end{pmatrix},\quad\text{where }
		X=\begin{pmatrix}
		x_1 & X_2\\
		X_3^T & Y
		\end{pmatrix},\quad\text{and }
		Y=\cir{(A_0,A_1,\ldots,A_{k-1})},
		\end{equation*}	
	and where $x_1\in R$, $A_i=\cir{\vct{a}_i}=\cir{(a_{i:0},a_{i:1},\ldots,a_{i:n-1})}\in R^n$ for $i\in[0\zint k-1]$ and $X_2=(\vct{x}_2,\vct{x}_2,\ldots,\vct{x}_2),X_3=(\vct{x}_3,\vct{x}_3,\ldots,\vct{x}_3)\in R^k$ such that $\vct{x}_2=(x_2,x_2,\ldots,x_2),\vct{x}_3=(x_3,x_3,\ldots,x_3)\in R^n$. Then $G$ is a generator matrix of a Hermitian self-dual $[2(kn+1),kn+1]$-code over $R$ if and only if
		\begin{empheq}[left=\empheqlbrace]{align*}
		1+x_1\hip{x_1}+knx_2\hip{x_2}&=0,\\
		x_1\hip{x_3}+x_2\sum_{i=0}^{k-1}\sum_{s=0}^{n-1}\hip{a_{i:s}}&=0,\\
		\sum_{i=0}^{k-1}\Theta(\vct{a}_i,\hip{\vct{a}_i},t)&=\begin{cases}
		-1-x_3\hip{x_3},&t=0,\\
		-x_3\hip{x_3},&t\in[1\zint\floor{n/2}],
		\end{cases}\\
		\sum_{i=0}^{k-1}\Theta(\vct{a}_{[i+j]_k},\hip{\vct{a}_i},0)&=-x_3\hip{x_3},\quad j\in[1\zint\floor{k/2},\\
		\sum_{i=0}^{k-1}\Theta(\vct{a}_{[i+j]_k},\hip{\vct{a}_i},t)&=-x_3\hip{x_3},\quad\begin{aligned}
		&j\in[1\zint k-1],\\
		&t\in[1\zint\floor{n/2}].
		\end{aligned}
		\end{empheq}		
	\end{thm}	
	
	\begin{proof}
		We know that $G$ is a generator matrix of a Hermitian self-dual $[2(kn+1),kn+1]$-code over $R$ if and only if $X\hip{X}{}^T=-I_{kn+1}$. We have
			\begin{align*}
				X\hip{X}{}^T&=
				\begin{pmatrix}
				x_1 & X_2\\
				X_3^T & Y
				\end{pmatrix}
				\begin{pmatrix}
				\hip{x_1} & \hip{X_2}\\
				\hip{X_3}{}^T & \hip{Y}
				\end{pmatrix}^T\\&=
				\begin{pmatrix}
				x_1 & X_2\\
				X_3^T & Y
				\end{pmatrix}
				\begin{pmatrix}
				\hip{x_1} & \hip{X_3}\\
				\hip{X_2}{}^T & \hip{Y}{}^T
				\end{pmatrix}\\&=	
				\begin{pmatrix}
				x_1\hip{x_1}+X_2\hip{X_2}{}^T & x_1\hip{X_3}+X_2\hip{Y}{}^T\\
				\hip{x_1}X_3^T+Y\hip{X_2}{}^T & X_3^T\hip{X_3}+Y\hip{Y}{}^T
				\end{pmatrix}\\&=
				\begin{pmatrix}
				Z_1 & Z_2\\
				Z_2^T & Z_3
				\end{pmatrix}
			\end{align*}
		so that $X\hip{X}{}^T=-I_{kn+1}$ if and only if $Z_1=-1$, $Z_2=\vct{0}$ and $Z_3=-I_{kn}$. It is easy to see that $Z_1=x_1\hip{x_1}+X_2\hip{X_2}{}^T=x_1\hip{x_1}+knx_2\hip{x_2}$, so we must have
			\begin{equation*}
			1+x_1\hip{x_1}+knx_2\hip{x_2}=0.
			\end{equation*}
		
		Let $X_2\hip{Y}{}^T=(\vct{u}_1,\vct{u}_2,\ldots,\vct{u}_k)$, where $\vct{u}_j=(u_{j:1},u_{j:2},\ldots,u_{j:n})\in R^n$, $\forall j\in[1\zint k]$. Then
			\begin{align*}
			Z_2&=x_1\hip{X_3}+X_2\hip{Y}{}^T\\
			&=x_1
			\begin{pmatrix}
			\hip{\vct{x}_3} & \hip{\vct{x}_3} & \cdots & \hip{\vct{x}_3} 
			\end{pmatrix}+
			\begin{pmatrix}
			\vct{u}_1 & \vct{u}_2 & \cdots & \vct{u}_3
			\end{pmatrix}\\&=
			\begin{pmatrix}
			x_1\hip{\vct{x}_3}+\vct{u}_1 & x_1\hip{\vct{x}_3}+\vct{u}_2 & \cdots & x_1\hip{\vct{x}_3}+\vct{u}_k
			\end{pmatrix}\\&=
			\begin{pmatrix}
			x_1\hip{\vct{x}_3}+u_{1:1} & x_1\hip{\vct{x}_3}+u_{1:2} & \cdots & x_1\hip{\vct{x}_3}+u_{k:n}
			\end{pmatrix}
			\end{align*}
		and so we see that $Z_2=\vct{0}$ if and only if $x_1\hip{\vct{x}_3}+u_{j:t}=0$, $\forall j\in[1\zint k]$ and $t\in[1\zint n]$. Since $Y=\cir{(A_0,A_1,\ldots,A_{k-1})}$ and $A_i=\cir{(a_{i:0},a_{i:1},\ldots,a_{i:n-1})}$, we observe that $\vct{u}_j=\vct{x}_2\sum_{i=0}^{k-1}\hip{A_i}{}^T$ so that $u_{j:t}=x_2\sum_{i=0}^{k-1}\sum_{s=0}^{n-1}\hip{a_{i:s}}$, $\forall j\in[1\zint k]$ and $t\in[1\zint n]$. Thus, we find that $Z_2=\vct{0}$ if and only if
			\begin{equation*}
			x_1\hip{x_3}+x_2\sum_{i=0}^{k-1}\sum_{s=0}^{n-1}\hip{a_{i:s}}=0.
			\end{equation*}	
		
		Finally, since
			\begin{equation*}
			X_3^T\hip{X_3}=
			\begin{pmatrix}
			x_3\hip{x_3} & x_3\hip{x_3} & \cdots & x_3\hip{x_3}\\
			x_3\hip{x_3} & x_3\hip{x_3} & \cdots & x_3\hip{x_3}\\
			\vdots & \vdots & \ddots & \vdots\\
			x_3\hip{x_3} & x_3\hip{x_3} & \cdots & x_3\hip{x_3}
			\end{pmatrix}
			\end{equation*}
		we have that $Z_3=-I_{kn}$ if and only if 
			\begin{equation*}
			Y\hip{Y}{}^T=
			\begin{pmatrix}
			-1-x_3\hip{x_3} & -x_3\hip{x_3} & \cdots & -x_3\hip{x_3}\\
			-x_3\hip{x_3} & -1-x_3\hip{x_3} & \cdots & -x_3\hip{x_3}\\
			\vdots & \vdots & \ddots & \vdots\\
			-x_3\hip{x_3} & -x_3\hip{x_3} & \cdots & -1-x_3\hip{x_3}
			\end{pmatrix}.\geteq\label{eqn-4}
			\end{equation*}
		
		By following the proof of Theorem \ref{thm-2}, we find that \eqref{eqn-4} is satisfied if and only if
			\begin{empheq}[left=\empheqlbrace]{align*}
			\sum_{i=0}^{k-1}\Theta(\vct{a}_i,\hip{\vct{a}_i},t)&=\begin{cases}
			-1-x_3\hip{x_3},&t=0,\\
			-x_3\hip{x_3},&t\in[1\zint\floor{n/2}],
			\end{cases}\\
			\sum_{i=0}^{k-1}\Theta(\vct{a}_{[i+j]_k},\hip{\vct{a}_i},0)&=-x_3\hip{x_3},\quad j\in[1\zint\floor{k/2},\\
			\sum_{i=0}^{k-1}\Theta(\vct{a}_{[i+j]_k},\hip{\vct{a}_i},t)&=-x_3\hip{x_3},\quad\begin{aligned}
			&j\in[1\zint k-1],\\
			&t\in[1\zint\floor{n/2}].
			\end{aligned}
			\end{empheq}				
	\end{proof}
	\begin{remark}
		The search field for Hermitian self-dual $[2(kn+1),kn+1]$-codes over $R$ constructed by Theorem \ref{thm-3} is of size $|R|^{kn+3}$.
	\end{remark}
\section{Results}

In this section, we apply the three techniques to construct best known quaternary Hermitian self-dual codes of lengths 24--40. We also apply the following well-known technique for constructing Hermitian self-dual codes referred to as the building-up construction.
	\begin{thm}\label{thm-4}\textup{(\cite{8})}
	Let $R$ be a commutative Frobenius ring with a fixed non-trivial involutory automorphism $\hip{\phantom{a}}:R\to R$. Let $G'$ be a generator matrix of a Hermitian self-dual $[2n,n]$-code $\mathcal{C}'$ over $R$ and let $\vct{r}_i$ denote the $i\nth$ row of $G'$. Let $\varepsilon\in R:\varepsilon\hip{\varepsilon}=-1$, $\vctg{\updelta}\in R^{2n}:\langle\vctg{\updelta},\vctg{\updelta}\rangle_H=-1$ and $\gamma_i=\langle\vct{r}_i,\vctg{\updelta}\rangle_H$ for $i\in[1\zint n]$. Then the matrix
		\begin{equation*}
		G=\begin{pmatrix}[cc|c]
		1 & 0 & \vctg{\updelta}\\\hline
		-\gamma_1 & \varepsilon\gamma_1 & \vct{r}_1\\
		-\gamma_2 & \varepsilon\gamma_2 & \vct{r}_2\\
		\vdots & \vdots & \vdots\\
		-\gamma_n & \varepsilon\gamma_n & \vct{r}_n
		\end{pmatrix}
		\end{equation*}
	is a generator matrix of a Hermitian self-dual $[2(n+1),n+1]$-code over $R$.
	\end{thm}
	
We conduct the search for these codes using MATLAB and determine their properties using Q-extension \cite{7}. Table \ref{table-1} gives the hexadecimal notation system we use to represent elements of $\FF_4$ and $\FF_4+u\FF_4$.

	\begin{table}[h!]
	\caption{Hexadecimal notation system for elements of $\FF_4$ and $\FF_4+u\FF_4$.}\label{table-1}
	\begin{tabular}{ccc}\midrule
	$\FF_4$ & $\FF_4+u\FF_4$ & Symbol\\\midrule
    $0$ & $0$ & \texttt{0}\\
	$1$ & $1$ & \texttt{1}\\
	$w$ & $w$ & \texttt{2}\\
	$1+w$ & $1+w$ & \texttt{3}\\
	$-$ & $u$ & \texttt{4}\\
	$-$ & $1+u$ & \texttt{5}\\
	$-$ & $w+u$ & \texttt{6}\\
	$-$ & $1+w+u$ & \texttt{7}\\
	$-$ & $wu$ & \texttt{8}\\
	$-$ & $1+wu$ & \texttt{9}\\
	$-$ & $w+wu$ & \texttt{A}\\
	$-$ & $1+w+wu$ & \texttt{B}\\
	$-$ & $u+wu$ & \texttt{C}\\
	$-$ & $1+u+wu$ & \texttt{D}\\
	$-$ & $w+u+wu$ & \texttt{E}\\
	$-$ & $1+w+u+wu$ & \texttt{F}\\\midrule
	\end{tabular}
	\end{table}
	
Using Theorems \ref{thm-1} and \ref{thm-2}, we find 347 and 9 Hermitian self-dual $[24,12,8]$-codes, respectively. Each code constructed by Theorem \ref{thm-2} is equivalent to one of those constructed by Theorem \ref{thm-1}. All of the constructed Hermitian self-dual $[24,12,8]$-codes possess weight enumerators which have been previously known to exist \cite{21,18,10,11,16}.  

Using Theorem \ref{thm-3}, we find Hermitian self-dual $[26,13,8]$-codes with 7 unique weight enumerators which have been previously known to exist. By applying Theorem \ref{thm-4} to the Hermitian self-dual codes of length 24 over $\FF_4$ constructed by Theorem \ref{thm-1}, we obtain Hermitian self-dual $[26,13,8]$-codes with 119 unique weight enumerators. Of these weight enumerators, 71 are new with respect to those reported in \cite{10,11,12}.

Using Theorem \ref{thm-1} and \ref{thm-2}, we find 3 and 1 Hermitian self-dual $[28,14,10]$-codes, respectively. The code constructed by Theorem \ref{thm-2} is equivalent one of those constructed by Theorem \ref{thm-1}. The number of inequivalent Hermitian self-dual $[28,14,10]$-codes we find does not improve on the current bound \cite{19}.

Using Theorem \ref{thm-2} and Theorem \ref{thm-3}, we find one Hermitian self-dual $[30,15,12]$-code. It is unknown if there exists more than one Hermitian self-dual $[30,15,12]$-code up to equivalence. 

Using Theorem \ref{thm-1}, we find 1715 inequivalent Hermitian self-dual $[32,16,10]$-codes each of which has one of 90 unique weight enumerators. Of these weight enumerators, 82 are new with respect to those reported in \cite{10,11,13}.

Using Theorem \ref{thm-2}, we find Hermitian self-dual $[36,18,12]$-codes with 2 unique weight enumerators which have been previously known to exist. Using Theorem \ref{thm-1}, we find Hermitian self-dual $[36,18,12]$-codes with 4 unique weight enumerators. Of these weight enumerators, 2 are new with respect to those reported in \cite{10,13}.

Using Theorem \ref{thm-3}, we find Hermitian self-dual $[38,19,12]$-codes with 13 unique weight enumerators. Of these weight enumerators, 1 is new with respect to those reported in \cite{10,13}.

Using Theorems \ref{thm-1} and \ref{thm-2}, we find Hermitian self-dual $[40,20,12]$-codes with 251 and 28 unique weight enumerators, respectively. Of these weight enumerators, 252 are new with respect to those reported in \cite{10,13}.
\subsection{New Type IV \BM{[26,13,8]}-Codes}

The weight enumerator of a quaternary Hermitian self-dual $[26,13,8]$-code is given by
	\begin{align*}
	W_{26}=1+\alpha x^8+(10725-5\alpha)x^{10}+\cdots,
	\end{align*}
where $\alpha\in\ZZ$. The existence of codes with weight enumerator $W_{26}$ has previously been determined for
	\begin{enumerate}[label=]
	\item $\alpha\in\{3z:z=39,\lb52,\lb70,\lb76,\lb78,\lb79,\lb81,\lb82,\lb84,\lb85,\lb88,\lb90,\lb91,\lb93,\lb94,\lb96,\lb97,\lb99,\lb100,\lb102,\lb103,\lb105,\lb106,\lb108,\lb109,\lb111,\lb112,\lb115,\lb118,\lb120,\lb124,\lb126,\lb127,\lb130,\lb133,\lb136,\lb138,\lb139,\lb142,\lb144,\lb145,\lb147,\lb148,\lb151,\lb154,\lb162,\lb163,\lb202,\lb225,\lb235,\lb442\}$
	\end{enumerate}	
(see \cite{10,11,12}).

We obtain 71 new optimal quaternary Hermitian self-dual codes of length 26 which have weight enumerator $W_{26}$ for
	\begin{enumerate}[label=]
	\item $\alpha\in\{3z:z=51,\lb54,\lb55,\lb57,\lb58,\lb60,\lb61,\lb63,\lb64,\lb66,\lb67,\lb69,\lb72,\lb73,\lb75,\lb87,\lb114,\lb117,\lb121,\lb123,\lb129,\lb132,\lb135,\lb141,\lb150,\lb153,\lb156,\lb157,\lb159,\lb160,\lb165,\lb166,\lb168,\lb169,\lb171,\lb172,\lb174,\lb175,\lb177,\lb178,\lb180,\lb181,\lb183,\lb184,\lb186,\lb187,\lb189,\lb190,\lb192,\lb193,\lb195,\lb196,\lb198,\lb199,\lb201,\lb205,\lb208,\lb210,\lb211,\lb213,\lb214,\lb217,\lb219,\lb220,\lb222,\lb243,\lb249,\lb265,\lb280,\lb283,\lb331\}$.
	\end{enumerate}	

Of the 71 new codes, 51 are constructed by first applying Theorem \ref{thm-1} to obtain codes of length 24 over $\FF_4$ (Table \ref{table-26-1}) to which we then apply Theorem \ref{thm-4} (Table \ref{table-26-2}) and similarly 20 are constructed by first applying Theorem \ref{thm-1} to obtain codes of length 12 over $\FF_4+u\FF_4$ (Table \ref{table-26-3}) to the image of which under $\varphi_{\FF_4+u\FF_4}$ we then apply Theorem \ref{thm-4} (Table \ref{table-26-4}). In Table \ref{table-26-2}, we only list a sample of 25 codes to save space.

	\begin{table}[h!]
	\caption{Codes of length 24 over $\FF_4$ from Theorem \ref{thm-1} to which we apply Theorem \ref{thm-4} to obtain new quaternary Hermitian self-dual $[26,13,8]$-codes.}\label{table-26-1}
	\begin{tabular}{cccccc}\midrule
	$\mathcal{C}_{24,i}'$ & $\lambda$ & $\mu$ & $\vct{a}$ & $\vct{b}$ & $\vct{c}$\\\midrule
	1 & \texttt{1} & \texttt{1} & \texttt{(000333)} & \texttt{(110101)} & \texttt{(311023)}\\
	2 & \texttt{1} & \texttt{1} & \texttt{(001320)} & \texttt{(102032)} & \texttt{(103123)}\\
	3 & \texttt{1} & \texttt{1} & \texttt{(013032)} & \texttt{(230111)} & \texttt{(121202)}\\
	4 & \texttt{1} & \texttt{1} & \texttt{(023100)} & \texttt{(120310)} & \texttt{(110111)}\\
	5 & \texttt{1} & \texttt{1} & \texttt{(111212)} & \texttt{(231320)} & \texttt{(023232)}\\
	6 & \texttt{1} & \texttt{1} & \texttt{(121303)} & \texttt{(112002)} & \texttt{(220222)}\\
	7 & \texttt{1} & \texttt{1} & \texttt{(220330)} & \texttt{(022200)} & \texttt{(121202)}\\
	8 & \texttt{1} & \texttt{1} & \texttt{(233301)} & \texttt{(013021)} & \texttt{(123120)}\\
	9 & \texttt{1} & \texttt{1} & \texttt{(311001)} & \texttt{(012300)} & \texttt{(213210)}\\
	10 & \texttt{1} & \texttt{1} & \texttt{(322322)} & \texttt{(013322)} & \texttt{(021133)}\\
	11 & \texttt{1} & \texttt{2} & \texttt{(023012)} & \texttt{(302221)} & \texttt{(020000)}\\
	12 & \texttt{1} & \texttt{3} & \texttt{(123322)} & \texttt{(312000)} & \texttt{(000020)}\\
	13 & \texttt{2} & \texttt{1} & \texttt{(313300)} & \texttt{(132033)} & \texttt{(201332)}\\
	14 & \texttt{2} & \texttt{2} & \texttt{(301311)} & \texttt{(023330)} & \texttt{(221330)}\\
	15 & \texttt{2} & \texttt{2} & \texttt{(312221)} & \texttt{(100032)} & \texttt{(022133)}\\
	16 & \texttt{3} & \texttt{1} & \texttt{(030131)} & \texttt{(010220)} & \texttt{(111110)}\\
	17 & \texttt{3} & \texttt{1} & \texttt{(032201)} & \texttt{(021322)} & \texttt{(010201)}\\
	18 & \texttt{3} & \texttt{1} & \texttt{(320011)} & \texttt{(332012)} & \texttt{(203112)}\\
	19 & \texttt{3} & \texttt{1} & \texttt{(333310)} & \texttt{(010233)} & \texttt{(021212)}\\
	20 & \texttt{3} & \texttt{3} & \texttt{(210202)} & \texttt{(233320)} & \texttt{(302231)}\\\midrule
	\end{tabular}
	\vspace*{-6pt}
	\end{table}

	\begin{table}[h!]
	\caption{Sample of new quaternary Hermitian self-dual $[26,13,8]$-codes from applying Theorem \ref{thm-4} to $\mathcal{C}_{24,j}'$ as given in Table \ref{table-26-1}.}\label{table-26-2}
	\begin{tabular}{cccccc}\midrule
	$\mathcal{C}_{26,i}$ & $\mathcal{C}_{24,j}'$ & $\varepsilon$ & $\vctg{\updelta}$ & $\alpha$ & $|\aut{\mathcal{C}_{26,i}}|$\\\midrule
	1 & 9 & \texttt{1} & \texttt{(100322012302332000223211)} & 153 & $2\cdot 3^{2}$\\
	2 & 14 & \texttt{3} & \texttt{(000000000000020012300312)} & 162 & $3$\\
	3 & 6 & \texttt{1} & \texttt{(113021030121032301013110)} & 165 & $2^{3}\cdot 3^{2}$\\
	4 & 17 & \texttt{2} & \texttt{(013011212320020111233221)} & 171 & $3$\\
	5 & 16 & \texttt{3} & \texttt{(202113112010011102220021)} & 183 & $3$\\
	6 & 20 & \texttt{2} & \texttt{(220131301302312010110301)} & 189 & $3$\\
	7 & 19 & \texttt{3} & \texttt{(223030120113111120013331)} & 192 & $3$\\
	8 & 15 & \texttt{2} & \texttt{(113102330210331320300120)} & 207 & $3$\\
	9 & 18 & \texttt{3} & \texttt{(322333200221101233333323)} & 216 & $3$\\
	10 & 13 & \texttt{2} & \texttt{(331113313132213232321201)} & 342 & $3$\\
	11 & 3 & \texttt{3} & \texttt{(220320132223013133231113)} & 351 & $3$\\
	12 & 7 & \texttt{3} & \texttt{(301102303100323210012113)} & 405 & $3$\\
	13 & 10 & \texttt{3} & \texttt{(212111031302032101021311)} & 450 & $3$\\
	14 & 8 & \texttt{1} & \texttt{(201132302202112022200032)} & 468 & $2\cdot 3$\\
	15 & 8 & \texttt{3} & \texttt{(202101112103220023301203)} & 471 & $3$\\
	16 & 8 & \texttt{2} & \texttt{(123130331000001322103113)} & 477 & $2\cdot 3$\\
	17 & 8 & \texttt{2} & \texttt{(220021120302220003032103)} & 480 & $3^{2}$\\
	18 & 8 & \texttt{1} & \texttt{(033203131223103211232233)} & 495 & $3$\\
	19 & 8 & \texttt{2} & \texttt{(231133303202110122202101)} & 498 & $3$\\
	20 & 8 & \texttt{3} & \texttt{(100203102201010012113013)} & 504 & $2^{2}\cdot 3$\\
	21 & 8 & \texttt{1} & \texttt{(232031003333223011211202)} & 507 & $3$\\
	22 & 8 & \texttt{1} & \texttt{(023133033111211002221031)} & 513 & $3$\\
	23 & 8 & \texttt{2} & \texttt{(301322110132221310102103)} & 516 & $3$\\
	24 & 8 & \texttt{3} & \texttt{(330021100133123231101222)} & 522 & $2\cdot 3$\\
	25 & 8 & \texttt{3} & \texttt{(101132103322123130231333)} & 525 & $2^{5}\cdot 3$\\\midrule
	\end{tabular}
	\vspace*{-6pt}
	\end{table}

	\begin{table}[h!]
	\caption{Codes of length 12 over $\FF_4+u\FF_4$ from Theorem \ref{thm-1} to which we apply Theorem \ref{thm-4} to obtain new quaternary Hermitian self-dual $[26,13,8]$-codes.}\label{table-26-3}
	\begin{tabular}{cccccc}\midrule
	$\mathcal{C}_{12,i}'$ & $\lambda$ & $\mu$ & $\vct{a}$ & $\vct{b}$ & $\vct{c}$\\\midrule
	1 & \texttt{1} & \texttt{2} & \texttt{(5B6)} & \texttt{(B68)} & \texttt{(5C4)}\\
	2 & \texttt{1} & \texttt{5} & \texttt{(6D3)} & \texttt{(70D)} & \texttt{(3E6)}\\
	3 & \texttt{2} & \texttt{5} & \texttt{(62F)} & \texttt{(BF0)} & \texttt{(691)}\\
	4 & \texttt{2} & \texttt{5} & \texttt{(7CD)} & \texttt{(EDA)} & \texttt{(ED2)}\\
	5 & \texttt{2} & \texttt{5} & \texttt{(93E)} & \texttt{(7C9)} & \texttt{(B5B)}\\
	6 & \texttt{2} & \texttt{5} & \texttt{(A2D)} & \texttt{(ACE)} & \texttt{(F84)}\\
	7 & \texttt{2} & \texttt{5} & \texttt{(F4E)} & \texttt{(7BD)} & \texttt{(D1B)}\\
	8 & \texttt{3} & \texttt{1} & \texttt{(19D)} & \texttt{(CAF)} & \texttt{(AE7)}\\
	9 & \texttt{3} & \texttt{1} & \texttt{(2B8)} & \texttt{(755)} & \texttt{(ABB)}\\
	10 & \texttt{5} & \texttt{3} & \texttt{(835)} & \texttt{(12F)} & \texttt{(2C8)}\\
	11 & \texttt{5} & \texttt{F} & \texttt{(077)} & \texttt{(A22)} & \texttt{(842)}\\
	12 & \texttt{A} & \texttt{1} & \texttt{(12A)} & \texttt{(4EE)} & \texttt{(B62)}\\
	13 & \texttt{A} & \texttt{5} & \texttt{(A92)} & \texttt{(AB0)} & \texttt{(1D7)}\\
	14 & \texttt{F} & \texttt{1} & \texttt{(FFA)} & \texttt{(28F)} & \texttt{(E95)}\\
	15 & \texttt{F} & \texttt{5} & \texttt{(F4A)} & \texttt{(B1E)} & \texttt{(EA9)}\\\midrule
	\end{tabular}
	\vspace*{-6pt}
	\end{table}
	
	\begin{table}[h!]
	\caption{New quaternary Hermitian self-dual $[26,13,8]$-codes from applying Theorem \ref{thm-4} to $\varphi_{\FF_4+u\FF_4}(\mathcal{C}_{12,j}')$ as given in Table \ref{table-26-3}.}\label{table-26-4}
	\begin{tabular}{cccccc}\midrule
	$\mathcal{C}_{26,i}$ & $\mathcal{C}_{12,j}'$ & $\varepsilon$ & $\vctg{\updelta}$ & $\alpha$ & $|\aut{\mathcal{C}_{26,i}}|$\\\midrule
	26 & 3 & \texttt{3} & \texttt{(312212221032322210211120)} & 174 & $3$\\
	27 & 7 & \texttt{1} & \texttt{(010003212121112212011232)} & 180 & $3$\\
	28 & 1 & \texttt{1} & \texttt{(120212301003313113031212)} & 198 & $3$\\
	29 & 15 & \texttt{3} & \texttt{(033103000010230311302131)} & 201 & $3$\\
	30 & 8 & \texttt{1} & \texttt{(002111202112303110121123)} & 219 & $3$\\
	31 & 14 & \texttt{2} & \texttt{(002200300110121012020033)} & 225 & $3$\\
	32 & 4 & \texttt{1} & \texttt{(130113020303033122322112)} & 261 & $3$\\
	33 & 6 & \texttt{3} & \texttt{(110121211101221203200312)} & 363 & $3$\\
	34 & 5 & \texttt{2} & \texttt{(030010332223122320131030)} & 369 & $3$\\
	35 & 13 & \texttt{3} & \texttt{(113112211221100311303233)} & 387 & $3$\\
	36 & 2 & \texttt{2} & \texttt{(303312023022323300120110)} & 396 & $3$\\
	37 & 10 & \texttt{1} & \texttt{(300013300333121231303101)} & 423 & $3$\\
	38 & 12 & \texttt{3} & \texttt{(221010231201303010022223)} & 459 & $3$\\
	39 & 9 & \texttt{3} & \texttt{(121020201211100310201232)} & 594 & $3$\\
	40 & 11 & \texttt{3} & \texttt{(022213232122010231300111)} & 630 & $2\cdot 3^{2}$\\
	41 & 10 & \texttt{3} & \texttt{(013330212033312303301232)} & 639 & $2^{5}\cdot 3$\\
	42 & 11 & \texttt{1} & \texttt{(323022331330202311021330)} & 660 & $2^{2}\cdot 3$\\
	43 & 10 & \texttt{3} & \texttt{(032211313030013312130200)} & 666 & $2^{3}\cdot 3^{2}$\\
	44 & 10 & \texttt{1} & \texttt{(211100302012031010313322)} & 795 & $2^{3}\cdot 3^{2}$\\
	45 & 10 & \texttt{1} & \texttt{(112113302311332030122211)} & 840 & $2^{3}\cdot 3^{4}$\\\midrule
	\end{tabular}
	\vspace*{-6pt}
	\end{table}
	
\subsection{New Type IV \BM{[32,16,10]}-Codes}

The weight enumerator of a quaternary Hermitian self-dual $[32,16,10]$-code is given by
	\begin{align*}
	W_{32}=1+\alpha x^{10}+(67704-7\alpha)x^{12}+\cdots,
	\end{align*}
where $\alpha\in\ZZ$. The existence of codes with weight enumerator $W_{32}$ has previously been determined for
	\begin{enumerate}[label=]
	\item $\alpha\in\{3z:z=290,\lb292,\lb296,\lb298,\lb301,\lb302,\lb304,\lb307,\lb308,\lb310,\lb311,\lb313,\lb316,\lb320,\lb322,\lb325,\lb328,\lb334,\lb400,\lb448,\lb464,\lb544,\lb592,\lb608,\lb640,\lb656\}$
	\end{enumerate}	
(see \cite{10,11,13}).
	
We obtain 82 new best known quaternary Hermitian self-dual codes of length 32 which have weight enumerator $W_{32}$ for
	\begin{enumerate}[label=]
	\item $\alpha\in\{3z:z=404,\lb416,\lb424,\lb431,\lb440,\lb466,\lb472,\lb488,\lb490,\lb494,\lb496,\lb500,\lb502,\lb505,\lb506,\lb508,\lb511,\lb512,\lb514,\lb517,\lb518,\lb520,\lb521,\lb523,\lb524,\lb526,\lb527,\lb529,\lb530,\lb532,\lb533,\lb535,\lb536,\lb538,\lb539,\lb541,\lb542,\lb545,\lb547,\lb548,\lb550,\lb551,\lb553,\lb554,\lb556,\lb557,\lb559,\lb560,\lb562,\lb563,\lb565,\lb566,\lb568,\lb569,\lb571,\lb572,\lb574,\lb575,\lb577,\lb578,\lb580,\lb581,\lb583,\lb584,\lb586,\lb587,\lb589,\lb590,\lb593,\lb596,\lb598,\lb599,\lb602,\lb604,\lb610,\lb611,\lb614,\lb616,\lb620,\lb622,\lb628,\lb632\}$.
	\end{enumerate}
	
Of the 82 new codes, 70 are constructed by applying Theorem \ref{thm-1} over $\FF_4$ (Table \ref{table-32-1}) and 12 are constructed by applying Theorem \ref{thm-1} over $\FF_4+u\FF_4$ (Table \ref{table-32-2}). In Table \ref{table-32-1}, we only list a sample of 25 codes to save space.
	
	\begin{table}[h!]
	\caption{Sample of new quaternary Hermitian self-dual $[32,16,10]$-codes from Theorem \ref{thm-1} over $\FF_4$.}\label{table-32-1}
	\begin{tabular}{cccccccc}\midrule
	$\mathcal{C}_{32,i}$ & $\lambda$ & $\mu$ & $\vct{a}$ & $\vct{b}$ & $\vct{c}$ & $\alpha$ & $|\aut{\mathcal{C}_{32,i}}|$\\\midrule
	1 & \texttt{1} & \texttt{3} & \texttt{(10302231)} & \texttt{(31100022)} & \texttt{(11212033)} & 1212 & $2\cdot 3^{2}$\\
	2 & \texttt{2} & \texttt{3} & \texttt{(32211223)} & \texttt{(31100102)} & \texttt{(13102211)} & 1272 & $2^{3}\cdot 3$\\
	3 & \texttt{2} & \texttt{2} & \texttt{(30102321)} & \texttt{(13100031)} & \texttt{(21112012)} & 1293 & $3^{2}$\\
	4 & \texttt{2} & \texttt{2} & \texttt{(10010230)} & \texttt{(11110231)} & \texttt{(10300010)} & 1398 & $3$\\
	5 & \texttt{2} & \texttt{3} & \texttt{(13320102)} & \texttt{(13222303)} & \texttt{(13013122)} & 1416 & $2^{4}\cdot 3$\\
	6 & \texttt{1} & \texttt{1} & \texttt{(03232020)} & \texttt{(31223133)} & \texttt{(32202231)} & 1464 & $2^{3}\cdot 3$\\
	7 & \texttt{1} & \texttt{3} & \texttt{(02111223)} & \texttt{(01032222)} & \texttt{(13010203)} & 1470 & $3$\\
	8 & \texttt{2} & \texttt{2} & \texttt{(32213330)} & \texttt{(23310210)} & \texttt{(11132220)} & 1482 & $3$\\
	9 & \texttt{1} & \texttt{1} & \texttt{(30121002)} & \texttt{(10003310)} & \texttt{(10300131)} & 1488 & $2^{3}\cdot 3$\\
	10 & \texttt{1} & \texttt{2} & \texttt{(13113012)} & \texttt{(33011101)} & \texttt{(20300113)} & 1506 & $3$\\
	11 & \texttt{3} & \texttt{1} & \texttt{(21003333)} & \texttt{(21202313)} & \texttt{(02013201)} & 1515 & $3$\\
	12 & \texttt{2} & \texttt{2} & \texttt{(03233000)} & \texttt{(32333201)} & \texttt{(10302000)} & 1524 & $3$\\
	13 & \texttt{2} & \texttt{1} & \texttt{(03022133)} & \texttt{(10220103)} & \texttt{(32202231)} & 1533 & $3$\\
	14 & \texttt{2} & \texttt{3} & \texttt{(31320303)} & \texttt{(23201323)} & \texttt{(31210211)} & 1536 & $2^{3}\cdot 3$\\
	15 & \texttt{2} & \texttt{2} & \texttt{(10312210)} & \texttt{(21303003)} & \texttt{(00203133)} & 1551 & $3$\\
	16 & \texttt{3} & \texttt{2} & \texttt{(03110212)} & \texttt{(01131112)} & \texttt{(10323223)} & 1560 & $2^{3}\cdot 3$\\
	17 & \texttt{3} & \texttt{1} & \texttt{(32021203)} & \texttt{(11213011)} & \texttt{(20300232)} & 1563 & $3$\\
	18 & \texttt{2} & \texttt{1} & \texttt{(02203110)} & \texttt{(23032301)} & \texttt{(20010221)} & 1569 & $3$\\
	19 & \texttt{2} & \texttt{1} & \texttt{(02321321)} & \texttt{(02032223)} & \texttt{(00202020)} & 1572 & $3$\\
	20 & \texttt{3} & \texttt{3} & \texttt{(23103311)} & \texttt{(00232322)} & \texttt{(23001013)} & 1578 & $3$\\
	21 & \texttt{2} & \texttt{2} & \texttt{(02122313)} & \texttt{(21213323)} & \texttt{(22203030)} & 1581 & $3$\\
	22 & \texttt{2} & \texttt{2} & \texttt{(23311302)} & \texttt{(03102312)} & \texttt{(00030201)} & 1587 & $3$\\
	23 & \texttt{3} & \texttt{1} & \texttt{(03203120)} & \texttt{(10223210)} & \texttt{(00010101)} & 1590 & $3$\\
	24 & \texttt{3} & \texttt{1} & \texttt{(02011312)} & \texttt{(03023210)} & \texttt{(23110113)} & 1596 & $3$\\
	25 & \texttt{2} & \texttt{2} & \texttt{(23010312)} & \texttt{(02113003)} & \texttt{(33001023)} & 1599 & $3$\\\midrule
	\end{tabular}
	\vspace*{-6pt}
	\end{table}
	
	\begin{table}[h!]
	\caption{New quaternary Hermitian self-dual $[32,16,10]$-codes from Theorem \ref{thm-1} over $\FF_4+u\FF_4$.}\label{table-32-2}
	\begin{tabular}{cccccccc}\midrule
	$\mathcal{C}_{32,i}$ & $\lambda$ & $\mu$ & $\vct{a}$ & $\vct{b}$ & $\vct{c}$ & $\alpha$ & $|\aut{\mathcal{C}_{32,i}}|$\\\midrule
	26 & \texttt{3} & \texttt{2} & \texttt{(9C33)} & \texttt{(3EF3)} & \texttt{(C188)} & 1248 & $2^{3}\cdot 3$\\
	27 & \texttt{5} & \texttt{F} & \texttt{(29F1)} & \texttt{(3429)} & \texttt{(4B16)} & 1320 & $2\cdot 3^{2}$\\
	28 & \texttt{F} & \texttt{3} & \texttt{(34CB)} & \texttt{(38A9)} & \texttt{(8A80)} & 1500 & $2\cdot 3$\\
	29 & \texttt{F} & \texttt{1} & \texttt{(ECD5)} & \texttt{(1AC4)} & \texttt{(3C0C)} & 1518 & $2\cdot 3$\\
	30 & \texttt{1} & \texttt{3} & \texttt{(5CC9)} & \texttt{(CFAB)} & \texttt{(48C1)} & 1542 & $2\cdot 3$\\
	31 & \texttt{F} & \texttt{F} & \texttt{(3368)} & \texttt{(DEEF)} & \texttt{(A31C)} & 1554 & $2\cdot 3$\\
	32 & \texttt{A} & \texttt{3} & \texttt{(C6BD)} & \texttt{(1A26)} & \texttt{(09AB)} & 1806 & $2\cdot 3$\\
	33 & \texttt{5} & \texttt{3} & \texttt{(A00A)} & \texttt{(75CD)} & \texttt{(040F)} & 1830 & $2\cdot 3$\\
	34 & \texttt{1} & \texttt{2} & \texttt{(62B3)} & \texttt{(D70F)} & \texttt{(5FA0)} & 1842 & $2\cdot 3$\\
	35 & \texttt{3} & \texttt{F} & \texttt{(5F19)} & \texttt{(CFE2)} & \texttt{(5241)} & 1860 & $2\cdot 3$\\
	36 & \texttt{F} & \texttt{1} & \texttt{(ADE2)} & \texttt{(E9C3)} & \texttt{(F0FF)} & 1866 & $2\cdot 3$\\
	37 & \texttt{3} & \texttt{A} & \texttt{(16AF)} & \texttt{(0BBE)} & \texttt{(FC5F)} & 1884 & $2\cdot 3$\\\midrule
	\end{tabular}
	\vspace*{-6pt}
	\end{table}

\subsection{New Type IV \BM{[36,18,12]}-Codes}

The weight enumerator of a quaternary Hermitian self-dual $[36,18,12]$-code is given by
	\begin{align*}
	W_{36}=1+\alpha x^{12}+(771120-12\alpha)x^{14}+\cdots,
	\end{align*}
where $\alpha\in\ZZ$. The existence of codes with weight enumerator $W_{36}$ has previously been determined for
	\begin{enumerate}[label=]
	\item $\alpha\in\{9z:z=2316,3196\}$
	\end{enumerate}	
(see \cite{10,13}). 

We obtain 2 new best known quaternary Hermitian self-dual codes of length 36 which have weight enumerator $W_{36}$ for
	\begin{enumerate}[label=]
	\item $\alpha\in\{9z:z=2172,2461\}$.
	\end{enumerate}

The new codes are constructed by applying Theorem \ref{thm-1} over $\FF_4$ (Table \ref{table-36-1}).

	\begin{table}[h!]
	\caption{New quaternary Hermitian self-dual $[36,18,12]$-codes from Theorem \ref{thm-1} over $\FF_4$.}\label{table-36-1}
	\begin{tabular}{cccccccc}\midrule
	$\mathcal{C}_{36,i}$ & $\lambda$ & $\mu$ & $\vct{a}$ & $\vct{b}$ & $\vct{c}$ & $\alpha$ & $|\aut{\mathcal{C}_{36,i}}|$\\\midrule
	1 & \texttt{3} & \texttt{2} & \texttt{(310331201)} & \texttt{(330313112)} & \texttt{(001133121)} & 19548 & $2\cdot 3^{3}$\\
	2 & \texttt{1} & \texttt{1} & \texttt{(021223233)} & \texttt{(301010120)} & \texttt{(323302310)} & 22149 & $2\cdot 3^{3}$\\\midrule
	\end{tabular}
	\vspace*{-6pt}
	\end{table}

\subsection{New Type IV \BM{[38,19,12]}-Codes}

The weight enumerator of a quaternary Hermitian self-dual $[38,19,12]$-code is given by
	\begin{align*}
	W_{38}=1+\alpha x^{12}+(430236-9\alpha)x^{14}+\cdots,
	\end{align*}
where $\alpha\in\ZZ$. The existence of codes with weight enumerator $W_{36}$ has previously been determined for
	\begin{enumerate}[label=]
	\item $\alpha\in\{3z:z=3249,\lb3279,\lb3354,\lb3483,\lb3534,\lb3714,\lb3795,\lb3800,\lb3831,\lb3897,\lb3927,\lb3987,\lb4383,\lb6327\}$
	\end{enumerate}	
(see \cite{10,13}). 

We obtain 1 new best known quaternary Hermitian self-dual code of length 38 which has weight enumerator $W_{38}$ for
	\begin{enumerate}[label=]
	\item $\alpha\in\{3z:z=3384\}$.
	\end{enumerate}
	
The new code is constructed by applying Theorem \ref{thm-3} over $\FF_4$ (Table \ref{table-38-1}).

	\begin{table}[h!]
	\caption{New quaternary Hermitian self-dual $[38,19,12]$-code from Theorem \ref{thm-3} over $\FF_4$.}\label{table-38-1}
	\begin{tabular}{cccccccccc}\midrule
	$\mathcal{C}_{38,i}$ & $k$ & $x_1$ & $x_2$ & $x_3$ & $\vct{a}_1$ & $\vct{a}_2$ & $\vct{a}_3$ & $\alpha$ & $|\aut{\mathcal{C}_{38,i}}|$\\\midrule
	1 & 3 & \texttt{3} & \texttt{2} & \texttt{1} & \texttt{(222010)} & \texttt{(012133)} & \texttt{(210331)} & 10152 & $2^{5}\cdot 3^{3}$\\\midrule
	\end{tabular}
	\vspace*{-6pt}
	\end{table}

\subsection{New Type IV \BM{[40,20,12]}-Codes}

The weight enumerator of a quaternary Hermitian self-dual $[40,20,12]$-code is given by
	\begin{align*}
	W_{40}=1+\alpha x^{12}+(232560-6\alpha)x^{14}+\cdots,
	\end{align*}
where $\alpha\in\ZZ$. The existence of codes with weight enumerator $W_{40}$ has previously been determined for
	\begin{enumerate}[label=]
	\item $\alpha\in\{3z:z=1387,\lb1425,\lb1470,\lb1510,\lb1570,\lb1580,\lb1590,\lb1650,\lb1690,\lb1710,\lb1720,\lb1730,\lb1740,\lb1750,\lb1800,\lb1810,\lb1824,\lb1850,\lb1860,\lb1870,\lb1881,\lb1919,\lb1938,\lb1960,\lb1990,\lb2010,\lb2033,\lb2070,\lb2110,\lb2210,\lb2261,\lb2310,\lb2370,\lb2584,\lb2940\}$
	\end{enumerate}	
(see \cite{10,13}). 

We obtain 252 new best known quaternary Hermitian self-dual codes of length 40 which have weight enumerator $W_{40}$ for
	\begin{enumerate}[label=]
	\item $\alpha\in\{3z:z=1265,\lb1285,\lb1330,\lb1340,\lb1345,\lb1350,\lb1360,\lb1365,\lb1380,\lb1385,\lb1390,\lb1395,\lb1400,\lb1405,\lb1410,\lb1415,\lb1420,\lb1430,\lb1435,\lb1440,\lb1445,\lb1450,\lb1455,\lb1460,\lb1465,\lb1475,\lb1480,\lb1485,\lb1490,\lb1495,\lb1500,\lb1505,\lb1515,\lb1520,\lb1525,\lb1530,\lb1535,\lb1540,\lb1545,\lb1548,\lb1550,\lb1555,\lb1560,\lb1565,\lb1575,\lb1585,\lb1595,\lb1598,\lb1600,\lb1605,\lb1610,\lb1615,\lb1620,\lb1624,\lb1625,\lb1630,\lb1631,\lb1635,\lb1640,\lb1644,\lb1645,\lb1655,\lb1656,\lb1660,\lb1662,\lb1665,\lb1666,\lb1667,\lb1670,\lb1674,\lb1675,\lb1680,\lb1681,\lb1682,\lb1685,\lb1686,\lb1688,\lb1692,\lb1695,\lb1699,\lb1700,\lb1702,\lb...,\lb1706,\lb1709,\lb1713,\lb1714,\lb1715,\lb1721,\lb...,\lb1728,\lb1731,\lb1735,\lb1736,\lb1737,\lb1741,\lb1742,\lb1745,\lb1746,\lb1748,\lb1749,\lb1754,\lb1755,\lb1756,\lb1758,\lb1759,\lb1760,\lb1765,\lb1766,\lb1769,\lb...,\lb1772,\lb1774,\lb1775,\lb1776,\lb1778,\lb1780,\lb1781,\lb1782,\lb1785,\lb1788,\lb1790,\lb1792,\lb...,\lb1795,\lb1805,\lb1812,\lb1813,\lb1815,\lb1820,\lb1822,\lb1825,\lb1828,\lb1830,\lb1835,\lb1840,\lb1845,\lb1854,\lb1855,\lb1856,\lb1858,\lb1865,\lb1873,\lb1875,\lb1876,\lb1880,\lb1885,\lb1890,\lb1895,\lb1897,\lb1900,\lb1905,\lb1910,\lb1915,\lb1920,\lb1925,\lb1930,\lb1935,\lb1940,\lb1945,\lb1950,\lb1955,\lb1965,\lb1970,\lb1972,\lb1975,\lb1980,\lb1985,\lb1995,\lb2000,\lb2005,\lb2015,\lb2020,\lb2025,\lb2029,\lb2030,\lb2035,\lb2040,\lb2045,\lb2050,\lb2055,\lb2060,\lb2065,\lb2075,\lb2080,\lb2085,\lb2090,\lb2095,\lb2100,\lb2105,\lb2115,\lb2120,\lb2125,\lb2130,\lb2135,\lb2140,\lb2145,\lb2150,\lb2155,\lb2160,\lb2165,\lb2170,\lb2175,\lb2180,\lb2185,\lb2190,\lb2195,\lb2200,\lb2205,\lb2220,\lb2230,\lb2240,\lb2245,\lb2250,\lb2270,\lb2275,\lb2280,\lb2285,\lb2290,\lb2300,\lb2320,\lb2325,\lb2330,\lb2340,\lb2350,\lb2355,\lb2360,\lb2380,\lb2400,\lb2410,\lb2430,\lb2440,\lb2490,\lb2495,\lb2540,\lb2550,\lb2580,\lb2600,\lb2660,\lb2705,\lb2740,\lb2760,\lb3310\}$.
	\end{enumerate}

Of the 252 new codes, 182 are constructed by applying Theorem \ref{thm-1} over $\FF_4$ (Table \ref{table-40-1}); 67 are constructed by applying Theorem \ref{thm-1} over $\FF_4+u\FF_4$ (Table \ref{table-40-2}) and 3 are constructed by applying Theorem \ref{thm-2} over $\FF_4$ (Table \ref{table-40-3}). In both Tables \ref{table-40-1} and \ref{table-40-2}, we only list a sample of 25 codes to save space.

	\begin{table}[h!]
	\caption{Sample of new quaternary Hermitian self-dual $[40,20,12]$-codes from Theorem \ref{thm-1} over $\FF_4$.}\label{table-40-1}
	\begin{tabular}{cccccccc}\midrule
	$\mathcal{C}_{40,i}$ & $\lambda$ & $\mu$ & $\vct{a}$ & $\vct{b}$ & $\vct{c}$ & $\alpha$ & $|\aut{\mathcal{C}_{40,i}}|$\\\midrule
	1 & \texttt{1} & \texttt{1} & \texttt{(0010301300)} & \texttt{(1212023113)} & \texttt{(1212101212)} & 3795 & $2\cdot 3\cdot 5$\\
	2 & \texttt{2} & \texttt{3} & \texttt{(3010333100)} & \texttt{(0331223323)} & \texttt{(3221002012)} & 3855 & $2^{2}\cdot 3\cdot 5$\\
	3 & \texttt{2} & \texttt{3} & \texttt{(3111233122)} & \texttt{(0231102110)} & \texttt{(0023020120)} & 3990 & $2^{2}\cdot 3\cdot 5$\\
	4 & \texttt{2} & \texttt{3} & \texttt{(1323200332)} & \texttt{(3112223320)} & \texttt{(3311322330)} & 4020 & $2\cdot 3\cdot 5$\\
	5 & \texttt{2} & \texttt{3} & \texttt{(2021122302)} & \texttt{(2232002130)} & \texttt{(3122232320)} & 4035 & $2^{2}\cdot 3\cdot 5$\\
	6 & \texttt{2} & \texttt{3} & \texttt{(0110332302)} & \texttt{(3311203021)} & \texttt{(1020013223)} & 4050 & $2\cdot 3\cdot 5$\\
	7 & \texttt{2} & \texttt{3} & \texttt{(2321203013)} & \texttt{(1002200000)} & \texttt{(0000000300)} & 4080 & $2^{2}\cdot 3\cdot 5$\\
	8 & \texttt{2} & \texttt{3} & \texttt{(3102032023)} & \texttt{(2232021033)} & \texttt{(1023001331)} & 4095 & $2^{2}\cdot 3\cdot 5$\\
	9 & \texttt{1} & \texttt{1} & \texttt{(0330131303)} & \texttt{(1231322020)} & \texttt{(0100323211)} & 4140 & $2\cdot 3\cdot 5$\\
	10 & \texttt{3} & \texttt{2} & \texttt{(2121010233)} & \texttt{(0302123032)} & \texttt{(1333232301)} & 4155 & $2\cdot 3\cdot 5$\\
	11 & \texttt{3} & \texttt{2} & \texttt{(0023012120)} & \texttt{(3320301220)} & \texttt{(0313120033)} & 4170 & $2\cdot 3\cdot 5$\\
	12 & \texttt{1} & \texttt{1} & \texttt{(1303123321)} & \texttt{(0220111331)} & \texttt{(1331302232)} & 4185 & $2^{2}\cdot 3\cdot 5$\\
	13 & \texttt{2} & \texttt{3} & \texttt{(1121131303)} & \texttt{(1011000231)} & \texttt{(2221210131)} & 4200 & $2\cdot 3\cdot 5$\\
	14 & \texttt{3} & \texttt{2} & \texttt{(0133033131)} & \texttt{(1302013013)} & \texttt{(2230021032)} & 4215 & $2\cdot 3\cdot 5$\\
	15 & \texttt{3} & \texttt{2} & \texttt{(1303120121)} & \texttt{(2102310210)} & \texttt{(1321321320)} & 4230 & $2^{2}\cdot 3\cdot 5$\\
	16 & \texttt{2} & \texttt{3} & \texttt{(0011032331)} & \texttt{(2301130211)} & \texttt{(0312131122)} & 4245 & $2^{2}\cdot 3\cdot 5$\\
	17 & \texttt{2} & \texttt{3} & \texttt{(2231132101)} & \texttt{(0021212121)} & \texttt{(2300302002)} & 4260 & $2\cdot 3\cdot 5$\\
	18 & \texttt{3} & \texttt{2} & \texttt{(1331212331)} & \texttt{(2223232310)} & \texttt{(3310223311)} & 4290 & $2^{2}\cdot 3\cdot 5$\\
	19 & \texttt{3} & \texttt{2} & \texttt{(1033123303)} & \texttt{(1332130232)} & \texttt{(0323333233)} & 4305 & $2\cdot 3\cdot 5$\\
	20 & \texttt{2} & \texttt{3} & \texttt{(1103013132)} & \texttt{(1010121102)} & \texttt{(3102323212)} & 4320 & $2\cdot 3\cdot 5$\\
	21 & \texttt{1} & \texttt{1} & \texttt{(2332231303)} & \texttt{(1111002223)} & \texttt{(0200131322)} & 4335 & $2\cdot 3\cdot 5$\\
	22 & \texttt{3} & \texttt{2} & \texttt{(3110020212)} & \texttt{(2032110332)} & \texttt{(0203003210)} & 4350 & $2\cdot 3\cdot 5$\\
	23 & \texttt{3} & \texttt{2} & \texttt{(0133213002)} & \texttt{(1121211002)} & \texttt{(1012333020)} & 4365 & $2\cdot 3\cdot 5$\\
	24 & \texttt{3} & \texttt{3} & \texttt{(0032013101)} & \texttt{(1120332311)} & \texttt{(0000000030)} & 4380 & $2^{2}\cdot 3\cdot 5$\\
	25 & \texttt{2} & \texttt{3} & \texttt{(0313223032)} & \texttt{(2132201232)} & \texttt{(0211213223)} & 4395 & $2\cdot 3\cdot 5$\\\midrule
	\end{tabular}
	\vspace*{-6pt}
	\end{table}
	
	\begin{table}[h!]
	\caption{Sample of new quaternary Hermitian self-dual $[40,20,12]$-codes from Theorem \ref{thm-1} over $\FF_4+u\FF_4$.}\label{table-40-2}
	\begin{tabular}{cccccccc}\midrule
	$\mathcal{C}_{40,i}$ & $\lambda$ & $\mu$ & $\vct{a}$ & $\vct{b}$ & $\vct{c}$ & $\alpha$ & $|\aut{\mathcal{C}_{40,i}}|$\\\midrule
	26 & \texttt{A} & \texttt{3} & \texttt{(11C05)} & \texttt{(E31C7)} & \texttt{(7D9F6)} & 4644 & $2\cdot 3$\\
	27 & \texttt{F} & \texttt{F} & \texttt{(69F28)} & \texttt{(66279)} & \texttt{(01084)} & 4794 & $2\cdot 3$\\
	28 & \texttt{1} & \texttt{3} & \texttt{(BD417)} & \texttt{(55615)} & \texttt{(0F0C4)} & 4872 & $2\cdot 3$\\
	29 & \texttt{A} & \texttt{3} & \texttt{(A4E4F)} & \texttt{(42F5E)} & \texttt{(B5B25)} & 4893 & $2\cdot 3$\\
	30 & \texttt{5} & \texttt{2} & \texttt{(AB405)} & \texttt{(FEECE)} & \texttt{(83444)} & 4932 & $2\cdot 3$\\
	31 & \texttt{5} & \texttt{2} & \texttt{(D12CE)} & \texttt{(2EA65)} & \texttt{(A8484)} & 4968 & $2\cdot 3$\\
	32 & \texttt{2} & \texttt{1} & \texttt{(8147A)} & \texttt{(5B707)} & \texttt{(91F91)} & 4986 & $2\cdot 3$\\
	33 & \texttt{1} & \texttt{2} & \texttt{(F11B1)} & \texttt{(A6C40)} & \texttt{(C0108)} & 4998 & $2\cdot 3$\\
	34 & \texttt{A} & \texttt{A} & \texttt{(DF57B)} & \texttt{(60252)} & \texttt{(1B31E)} & 5001 & $2\cdot 3$\\
	35 & \texttt{A} & \texttt{2} & \texttt{(EF44A)} & \texttt{(68E3B)} & \texttt{(04140)} & 5022 & $2\cdot 3$\\
	36 & \texttt{F} & \texttt{2} & \texttt{(4ED9F)} & \texttt{(84931)} & \texttt{(ABD1F)} & 5043 & $2\cdot 3$\\
	37 & \texttt{2} & \texttt{A} & \texttt{(842AA)} & \texttt{(517D0)} & \texttt{(04380)} & 5046 & $2\cdot 3$\\
	38 & \texttt{2} & \texttt{1} & \texttt{(393D8)} & \texttt{(1C466)} & \texttt{(5DFD5)} & 5058 & $2\cdot 3$\\
	39 & \texttt{3} & \texttt{A} & \texttt{(0DCDD)} & \texttt{(DDE45)} & \texttt{(A3267)} & 5064 & $2\cdot 3$\\
	40 & \texttt{3} & \texttt{A} & \texttt{(86078)} & \texttt{(529FB)} & \texttt{(AF92E)} & 5076 & $2\cdot 3$\\
	41 & \texttt{2} & \texttt{5} & \texttt{(C163C)} & \texttt{(82898)} & \texttt{(E55D1)} & 5097 & $2\cdot 3$\\
	42 & \texttt{F} & \texttt{2} & \texttt{(F7ABF)} & \texttt{(4C89B)} & \texttt{(B1AD5)} & 5106 & $2\cdot 3$\\
	43 & \texttt{2} & \texttt{1} & \texttt{(A8806)} & \texttt{(4646D)} & \texttt{(D35D5)} & 5109 & $2\cdot 3$\\
	44 & \texttt{5} & \texttt{1} & \texttt{(16C69)} & \texttt{(73F32)} & \texttt{(D59B9)} & 5112 & $2\cdot 3$\\
	45 & \texttt{5} & \texttt{F} & \texttt{(A8992)} & \texttt{(4A514)} & \texttt{(3E272)} & 5118 & $2\cdot 3$\\
	46 & \texttt{F} & \texttt{3} & \texttt{(09409)} & \texttt{(2F629)} & \texttt{(7E37E)} & 5127 & $2\cdot 3$\\
	47 & \texttt{A} & \texttt{2} & \texttt{(8FF41)} & \texttt{(3FC7D)} & \texttt{(9179E)} & 5139 & $2\cdot 3$\\
	48 & \texttt{F} & \texttt{2} & \texttt{(6C1A6)} & \texttt{(BFC90)} & \texttt{(96D1E)} & 5142 & $2\cdot 3$\\
	49 & \texttt{1} & \texttt{2} & \texttt{(AB779)} & \texttt{(7ED06)} & \texttt{(291EF)} & 5163 & $2\cdot 3$\\
	50 & \texttt{3} & \texttt{5} & \texttt{(1425F)} & \texttt{(E8AD4)} & \texttt{(00C18)} & 5166 & $2\cdot 3$\\\midrule
	\end{tabular}
	\vspace*{-6pt}
	\end{table}
	
	\begin{table}[h!]
	\caption{New quaternary Hermitian self-dual $[40,20,12]$-codes from Theorem \ref{thm-3} over $\FF_4$.}\label{table-40-3}
	\begin{tabular}{cccccccc}\midrule
	$\mathcal{C}_{40,i}$ & $k$ & $\lambda$ & $\mu$ & $\vct{a}_1$ & $\vct{a}_2$ & $\alpha$ & $|\aut{\mathcal{C}_{40,i}}|$\\\midrule
	51 & 2 & \texttt{1} & \texttt{3} & \texttt{(3212220310)} & \texttt{(2302200133)} & 5760 & $2^{3}\cdot 3\cdot 5$\\
	52 & 2 & \texttt{1} & \texttt{2} & \texttt{(2213113102)} & \texttt{(2022022132)} & 5910 & $2^{3}\cdot 3\cdot 5$\\
	53 & 2 & \texttt{2} & \texttt{1} & \texttt{(2121231110)} & \texttt{(0020102020)} & 6660 & $2^{4}\cdot 3\cdot 5$\\\midrule
	\end{tabular}
	\vspace*{-6pt}
	\end{table}

\section{Conclusion}

In this work, we presented three techniques for constructing Hermitian self-dual codes using $\lambda$-circulant matrices. The first technique was derived as the Hermitian analogue of a modified four circulant technique for constructing self-dual codes given in \cite{17}. The other two techniques were derived as the Hermitian analogues of the techniques for constructing self-dual codes given in \cite{6}. We proved the necessary conditions for these techniques to produce Hermitian self-dual codes using a specialised mapping $\Theta$ which is inherently associated with $\lambda$-circulant matrices. We explored the ability of these techniques by implementing them to construct best known and optimal quaternary Hermitian self-dual codes which were previously not known to exist. In particular, we were able to construct new optimal codes of length 26 and many best known codes of lengths 32, 36, 38 and 40.

The performance of the first technique in searching for Hermitian self-dual codes was far superior to that of the other techniques. However, the first technique is restricted in that it can only be used to construct codes of length $4n$. The main advantage of using the other techniques is being able to construct codes of a greater variety of lengths. On the other hand, the efficiency of the second and third techniques drops drastically as length increases.

We were unable to improve on the best known minimum distance of quaternary Hermitian self-dual codes for lengths 32--40. Due to computational limitations, we did not investigate constructing codes of lengths greater than 40. Moreover, the codes we did construct were obtained by random searches alone. However, as in \cite{17}, the implementation of the mapping $\Theta$ enabled us to reduce the number of computations required in our algorithms and hence increase the rate at which we could find Hermitian self-dual codes.

A suggestion for future work could be attempting to construct codes of lengths greater than 40 or possibly even investigate feasible ways of classifying the codes yielded by the techniques for reasonably small lengths. Another suggestion would be to consider different families of unitary matrices for $C$.

\bibliographystyle{plainnat}
\bibliography{paper2}
\end{document}